\documentclass[a4paper,11pt]{article}
\pdfoutput=1

\usepackage[utf8]{inputenc}
\usepackage{geometry}
\usepackage{graphicx}
\usepackage{amsmath}
\usepackage{amsfonts,textcomp}
\usepackage{amssymb}
\usepackage{amsthm}
\usepackage{bbold}
\usepackage{subfig}
\usepackage{url}
\usepackage{enumitem}
\usepackage{theoremref}

\PassOptionsToPackage{table}{xcolor}
\usepackage{tikz}
\usetikzlibrary{arrows,calc,plotmarks,positioning}
\tikzset{>=angle 60}
\usepackage[colorlinks=true, linkcolor=blue,citecolor=blue, urlcolor=blue]{hyperref}

\makeatletter
\newcommand{\customlabel}[2]{%
   \protected@write \@auxout {}{\string \newlabel {#1}{{#2}{\thepage}{#2}{#1}{}} }%
   \hypertarget{#1}{#2}
}
\makeatother

\newcommand{\tn}[1]{\textnormal{#1}}
\newcommand{\mc}[1]{\mathcal{#1}}

\newcommand{\R}{\mathbb{R}}

\newcommand{\Z}{\mathbb{Z}}

\newcommand{\T}{\top}

\newcommand{\la}{\left \langle}
\newcommand{\ra}{\right \rangle}

\newcommand{\ol}[1]{\overline{#1}}

\DeclareMathOperator{\id}{id}

\DeclareMathOperator{\spt}{spt}

\newcommand{\eqdef}{\stackrel{\mbox{\upshape\tiny def.}}{=}}

\theoremstyle{plain}
\newtheorem{theorem}{Theorem}[section]

\newtheorem{proposition}[theorem]{Proposition}
\newtheorem{corollary}[theorem]{Corollary}
\newtheorem{algorithm}[theorem]{Algorithm}

\theoremstyle{definition}
\newtheorem{definition}[theorem]{Definition}

\newtheorem{remark}[theorem]{Remark}

\newcommand{\neigh}{\mc{N}}
\newcommand{\shield}{S}
\newcommand{\prob}{\mc{P}}

\newcommand{\Rinf}{{\ol{\R}}}

\newcommand{\hpartX}{\mc{X}}
\newcommand{\hpartY}{\mc{Y}}
\newcommand{\hcellX}{\mathbf{x}}
\newcommand{\hcellY}{\mathbf{y}}
\newcommand{\rep}{\tn{rep}}
\newcommand{\rad}{\tn{rad}}
\newcommand{\sphere}{{\mc{S}_2}}
\newcommand{\kreis}{{\mc{S}_1}}

\newcommand{\cGeo}{{c_{\tn{geo}}}}
\newcommand{\cNoise}{{c_{\tn{n}}}}
\newcommand{\cLip}{{c_{\tn{L}}}}
\newcommand{\vol}{\tn{vol}}

\newcommand{\funcShield}{\tn{\texttt{shield}}}
\newcommand{\funcSolve}{\tn{\texttt{solveLocal}}}
\newcommand{\funcSolveShortCuts}{\tn{\texttt{solveSparse}}}
\newcommand{\funcSolveMultiScale}{\tn{\texttt{solveMultiScale}}}
\newcommand{\funcSearch}{\tn{\texttt{searchTree}}}

\numberwithin{equation}{section}

\title{A Sparse Multi-Scale Algorithm for Dense Optimal Transport}
\author{Bernhard Schmitzer\\[2mm]
	CEREMADE, Universit\'e Paris-Dauphine \\
	\texttt{\small schmitzer@ceremade.dauphine.fr}}
\date{\today}

\begin{document}
\maketitle

\begin{abstract}
	Discrete optimal transport solvers do not scale well on dense large problems since they do not explicitly exploit the geometric structure of the cost function.
	In analogy to continuous optimal transport, we provide a framework to verify global optimality of a discrete transport plan locally.
	This allows the construction of an algorithm to solve large dense problems by considering a sequence of sparse problems instead. The algorithm lends itself to being combined with a hierarchical multi-scale scheme. Any existing discrete solver can be used as internal black-box.
	We explicitly describe how to select the sparse sub-problems for several cost functions, including the noisy squared Euclidean distance.
	A significant reduction of run-time and memory requirements is observed.
\end{abstract}

\tableofcontents


\section{Introduction}

\subsection{Background and Motivation}
\label{sec:IntroBackground}
Optimal transport (OT) is a classical optimization problem dating back to the seminal work of Monge and Kantorovich. Over the past decades it has been studied in great detail (see for example \cite{Villani-OptimalTransport-09} for a comprehensive monograph and some historical context and also \cite{Ambrosio2013,Santambrogio-OTAM} for helpful introductions to the subject).
An important step was the polar factorization theorem \cite{MonotoneRerrangement-91} on $\R^n$ for the cost being the squared Euclidean distance. Since then, this result has been generalized to other convex functions on $\R^n$ \cite{McCannGangboOTGeometry1996}, to the squared geodesic distance on Riemannian manifolds \cite{McCannPolarManifold2001} and more general costs on Riemannian manifolds \cite{OTManifoldLagrangian2007}.

OT is also a successfully and widely applied tool in image processing, computer vision and statistics (e.g.~\cite{RubnerEMD-IJCV2000,GeodesicShapeMassTransport-10,OptimalTransportTangent2012,
	wassersteinRegularization2011,RumpfGeneralizedOT2014,Steidl-OT-RGB-2015}).
However, it is computationally more costly than `simple' similarity measures such as $L^p$-distances or Bregman divergences. Consequently there is a need for efficient solvers.

Broadly speaking there are two classes of solvers:
There are discrete (combinatorial) algorithms based on the finite dimensional linear programming formulation, such as the Hungarian method \cite{KuhnHungarianMethod}, the auction algorithm \cite{Bertsekas-AuctionAlgorithm1979} (a parallelized GPU implementation is described in \cite{GraphMatchingGPU-EMMCVPR-2009}), the network simplex \cite{NetworkFlows1993} and more (e.g.~\cite{GoldbergCostScaling1990}).
They work for (almost) arbitrary cost functions, and are typically numerically robust w.r.t.\ input data regularity.
They do not scale well for large, dense problems however, because the geometric structure of the cost function is not used.
Alternatively, there are continuous solvers, based on the polar factorization theorem and the Monge-Amp\`ere equation (e.g.~\cite{OptimalTransportWarping,BrenierMap-10,ObermanMongeAmpere2014}).
These need not handle the full product space, but work directly with a transport map and thus can solve large problems more efficiently.
But they only apply to a restricted family of cost functions (most prominently the squared Euclidean distance) and they are numerically more subtle (e.g.\ involving the Jacobian of the transport map), thus requiring some data regularity.
The celebrated fluid-dynamics formulation \cite{BenamouBrenier2000} is more flexible but introduces the additional cost of a time-dimension.

For the particular case where the cost is a metric, there are specialized efficient solvers (see for example \cite{TreeEMD2007}).
In addition, a wide range of approximate methods has been applied: wavelets \cite{linearApproxMassTransportCVPR2008}, cost function thresholding \cite{Pele2009}, tangent space approximation \cite{OptimalTransportTangent2012} and entropic smoothing \cite{Cuturi2013,BenamouIterativeBregman2015} among others.

Multi-scale schemes have been proposed to accelerate exact solvers (\cite{MultiscaleTransport2011,SchmitzerSchnoerr-SSVM2013,Schmitzer-SSVM2015,ObermanOptimalTransportationLP2015}).
The original problem is approximated by a sequence of successively coarser problems. Starting from the coarsest resolution, the optimal coupling at a given scale will then provide a good initialization for the subsequent finer level.
However, \cite{MultiscaleTransport2011} is limited to the case of the squared Euclidean distance. The algorithm in \cite{SchmitzerSchnoerr-SSVM2013} only uses the geometric structure of the cost implicitly by keeping the problem sparse via hierarchical consistency checks, requiring low level adaptions of the algorithm.
The scheme presented in \cite{ObermanOptimalTransportationLP2015} is based on similar intuition as this article and works very well in practice but does not provide a rigorous framework for verifying global optimality (see Sect.\ \ref{sec:Conclusion} for a discussion).

So there is still a need for efficient discrete exact solvers that are more flexible than the continuous solvers (both in terms of cost functions and measure regularity), but which are still able to exploit the geometric structure of the cost function. Such an algorithm has been presented in \cite{Schmitzer-SSVM2015}, of which the present article is an extension.

\subsection{Outline and Contribution}
An important feature of continuous solvers is that under suitable conditions optimality of the transport plan can be verified by a local criterion: the transport map is the gradient of a convex function.
However, discrete solvers must check optimality globally (e.g.~all dual constraints must be verified).
In this paper we develop a framework for the discrete setting to mimic the locality property of the continuum. Local then means that we only need to look at a sparse sub-problem, determined by the transport plan and the cost.

In Section \ref{sec:OTBackground} we establish notation and briefly recall some basic properties of discrete optimal transport.
In Section \ref{sec:ShortCuts}, after gathering some intuition from the continuous setting, we develop a \emph{rigorous discrete framework} for inferring global optimality of a coupling from local optimality on a suitable sparse sub-problem. For this we introduce the notion of \emph{`shielding neighbourhoods'}.
Based on these results, in Section \ref{sec:Algorithm}, we design an algorithm that solves a dense problem via a sequence of sparse sub-problems. \emph{Convergence} of the algorithm and \emph{global optimality} of the resulting transport plan are proved. \emph{Any discrete OT solver can be used as internal sub-routine.} We propose to combine this algorithm with a hierarchical multi-scale scheme to obtain good initializations and consequently low running-times.
Section \ref{sec:Shielding} is devoted to an important component of the algorithm: the efficient construction of sparse shielding neighbourhoods, exploiting the geometry of the underlying cost-function. We discuss several types of costs on $\R^n$ and also provide an example for the sphere to underline the generality of the concept. A particularly relevant case is the squared Euclidean distance over Cartesian grids which allows the highest acceleration (Remark \ref{rem:ShieldingCartesianGrids}).
It is shown that unlike standard continuous solvers, the discrete method can \emph{tolerate certain types of noise and distortions} of the cost.
In Section \ref{sec:Numerics} a series of numerical experiments is presented to demonstrate the efficiency of the scheme. We observe speed-ups of up to \emph{two orders of magnitude}, depending on the problem class, and reduction in memory requirements by up to three orders with \emph{state-of-the-art solver software} as compared to naively solving the dense problem, thus empirically verifying the efficiency of the multi-scale scheme. The test problems involved both smooth as well as locally concentrated measures and both `clean' and noisy costs, thus indicating a wide range of practical applicability.
A concluding discussion is provided in Section \ref{sec:Conclusion}.

Compared to \cite{Schmitzer-SSVM2015}, in this article we provide more details on the multi-scale scheme (Sect.\ \ref{sec:AlgorithmHierarchy}) and explain how the algorithm can be applied to more general types of problems (Sect.\ \ref{sec:Shielding}): point clouds beyond regular Cartesian grids and cost functions besides the squared Euclidean distance are considered. Moreover, the numerical experiments have been extended (Sect. \ref{sec:Numerics}).


\section{Background on Optimal Transport}
\label{sec:OTBackground}
\paragraph{Notation.}
	For measure spaces $A$ and $B$ denote by $\prob(A)$ the space of probability measures over $A$. For a measurable map $f: A \rightarrow B$ and a measure $\mu \in \prob(A)$ we denote by $f_\sharp \mu \in \prob(B)$ the \emph{push-forward} of $\mu$ given by $f_\sharp \mu(\sigma) = \mu(f^{-1}(\sigma))$ for measurable $\sigma \subset B$.

	For a discrete finite set $A$ we write $|A|$ for its cardinality. For a measure $\mu \in \prob(A)$ its support is defined by $\spt \mu = \{ a \in A \colon \mu(a) > 0 \}$. For singletons we often just write $\mu(a) \eqdef \mu(\{a\})$ for $a \in A$.
	Write $\Rinf = \R \cup \{\infty\}$.
	The space of $\Rinf$-valued functions over $A$ is identified with $\Rinf^{|A|}$ where we index the dimensions by elements of $A$. We write $2^A$ for the power set of $A$.
	
	For a convex function $h : \R^n \rightarrow \Rinf$ we denote by $\partial h(x)$ its sub-differential at $x$.

\paragraph{Discrete Optimal Transport.}
For two discrete finite sets $X$, $Y$ and two probability measures $\mu \in \prob(X)$, $\nu \in \prob(Y)$ the set of couplings is given by
\begin{align}
	\Pi(\mu,\nu) = \left\{ \pi \! \in \prob(X \!\times\! Y) \colon
		\pi(\{x\} \!\times\! Y) = \mu(x),\,
		\pi(X \!\times\! \{y\}) = \nu(y)\,
		\,\forall\, x \in X, y \in Y
		\right\}\,.
\end{align}
For a cost function $c: X \times Y \rightarrow \Rinf$ the optimal transport problem consists of finding the coupling with minimal total transport cost:
\begin{align}
	\label{eq:OT}
	\min_{\pi \in \Pi(\mu,\nu)} C(\pi)
	\qquad \tn{with} \qquad
	C(\pi) = \sum_{(x,y) \in X \times Y} c(x,y)\,\pi(x,y)
\end{align}
The problem is called feasible if its optimal value is finite.
We call \eqref{eq:OT} the \emph{dense} or \emph{full problem}. For some $\neigh \subset X \times Y$ we also consider problem \eqref{eq:OT} subject to the additional constraint $\spt \pi \subset \neigh$, which we call the problem \emph{restricted to $\neigh$}. We call $\neigh$ a \emph{neighbourhood} and say $\neigh$ is feasible when the corresponding problem is feasible. The \emph{set of all possible neighbourhoods} is given by $2^{X \times Y}$ (not all of them being feasible).
We will call $\pi$ a local optimizer w.r.t.~$\neigh$ if it solves the corresponding restricted problem.

The dual problem to \eqref{eq:OT} is given by (\cite[Chapter 5]{Villani-OptimalTransport-09}, see also \cite[Chapters 4 and 7.8]{BertsimasLP})
\begin{subequations}
\label{eq:OTDual}
\begin{align}
	\max_{(\alpha,\beta) \in (\R^{|X|},\R^{|Y|})} \sum_{x \in X} \alpha(x)\,\mu(x) + \sum_{y \in Y} \beta(y)\,\nu(y) \\
	\label{eq:OTDualConstraints}
	\text{subject to} \quad \alpha(x) + \beta(y) \leq c(x,y) \quad \tn{for all} \quad (x,y) \in X \times Y\,.
\end{align}
\end{subequations}

		The relation between any primal and dual optimizers $\pi$ and $(\alpha,\beta)$ of the same transport problem is
\begin{equation}
	\label{eq:PDCondition}
	\pi(x,y) > 0 \qquad \Rightarrow \qquad \alpha(x) + \beta(y) = c(x,y)\,.
\end{equation}

		Restricting the primal problem to $\neigh$ corresponds to only enforcing the dual constraints \eqref{eq:OTDualConstraints} on $\neigh$. Analogously we speak of local dual optimizers $(\alpha,\beta)$ w.r.t~$\neigh$. If $(\pi,(\alpha,\beta))$ are local primal and dual optimizers and $(\alpha,\beta)$ satisfy \eqref{eq:OTDualConstraints} on $X \times Y$, then one has found optimizers for the full problem.

One goal of this paper is to find suitable small subsets $\neigh$ such that the local optimizers $(\pi,(\alpha,\beta))$ w.r.t.~$\neigh$ are also optimal for the full problem.


\section{Optimal Transport and Short-Cuts}
\label{sec:ShortCuts}

\subsection{Intuition from the Continuous Case}
\label{sec:ShortCutsIntuition}
The discrete algorithm we present in this article is inspired by continuous optimal transport. Let us therefore recall some well-known results from the continuous setting.

Let $\mu$, $\nu$ be Lebesgue absolutely continuous measures on $\R^n$ with compact, convex support. Consider the continuous optimal transport problem between $\mu$ and $\nu$ w.r.t.~the cost $c(x,y) = |x-y|^2$.
Then by virtue of Brenier's celebrated polar factorization theorem \cite{MonotoneRerrangement-91} we know 
that the optimal coupling is induced by a map which is the gradient of a convex function. Conversely, when a transport map is shown to be the gradient of a convex function then it is optimal.

Let now $T$ be any transport map, $T_{\sharp} \mu = \nu$, with induced coupling $\pi = (\id,T)_{\sharp} \mu$ (see for example \cite[Def.~1.2]{Villani-OptimalTransport-09}). For simplicity let $T$ be a homeomorphism. We want to verify optimality of $T$.
Let $\{U_i\}_i$ be an open covering of $\spt \mu$. Then $\{V_i\}_i$ with $V_i = T(U_i)$ is an open covering of $\spt \nu$. Let $\mu|_{U_i}$ and $\nu|_{V_i}$ be the restrictions of the measure $\mu$ to $U_i$ and $\nu$ to $V_i$. Then $T$ is also a transport map between $\mu|_{U_i}$ and $\nu|_{V_i}$ for all $i$.
If $T$ is optimal for each restricted problem on $U_i \times V_i$ then optimality for the whole problem follows: when we know that $T$ is the gradient of a convex function on each $U_i$, by convexity of $\spt \mu$ it follows that $T$ is the gradient of a convex function on $\spt \mu$ and thus is the optimal transport map.
Since the patches $U_i$ can be made arbitrarily small, optimality of a coupling $\pi$ can be verified on an arbitrary small open environment of $\spt \pi$ on $(\R^n)^2$. This is illustrated in Fig.~\ref{fig:IllustrationContinuum}.

The \emph{Monge property} \cite{ReviewMongeMatrix-96} is a simple discrete analogy in one dimension: one only needs to check whether two neighbours in $X$ can save costs by swapping mass. If there are no such neighbours then the coupling is optimal.
In this paper we strive to find a discrete equivalent for higher-dimensional problems.
We will return to this discussion for a brief comparison in Sect.~\ref{sec:ShieldingSqrEuclidean}, Remark \ref{rem:ShieldingEuclideanContinuous}.

\begin{figure}
	\centering%
	\subfloat[]{%
		\includegraphics{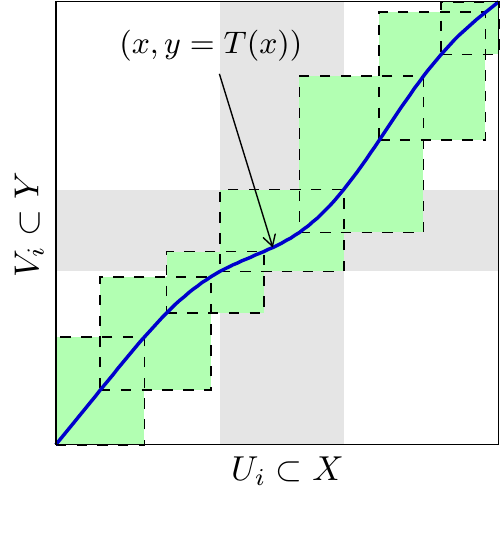}%
		\label{fig:IllustrationContinuum}%
	}%
	\hskip 2cm%
	\subfloat[]{%
		\includegraphics{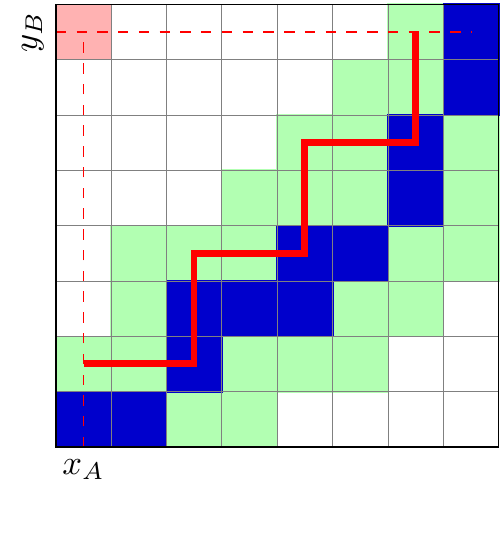}%
		\label{fig:IllustrationShortCut}%
	}%
	\caption{{\protect \subref{fig:IllustrationContinuum}}: For absolutely continuous measures on $X=Y=\R^n$ and the squared Euclidean distance as cost, the optimal transport plan will live on the graph of a map $T : X \rightarrow Y$, which is the gradient of a convex potential $\varphi : X \mapsto \R$. In the continuous case it suffices to check whether $T$ is optimal on each of the patches $U_i \times V_i$. Global optimality then follows. %
	{\protect \subref{fig:IllustrationShortCut}}: As a discrete analogy we introduce the concept of \emph{short-cuts}. The dual constraint at $(x_A,y_B)$ is implied by combining a suitable sequence (solid red line) of local constraints in $\neigh$ (green) and active constraints in $\spt \pi$ (blue).}
\end{figure}

\subsection{Short-Cuts}
Now we introduce the concept of short-cuts, a tool to temporarily remove constraints from the dual problem: dual constraints for which a short-cut exists need no longer be checked (see Fig.~\ref{fig:IllustrationShortCut} for an illustration).

\begin{definition}[Short-Cut]
	For a neighbourhood $\neigh \subset X \times Y$ and a coupling $\pi$ with $\spt \pi \subset \neigh$ let $((x_2,y_2),\ldots,(x_{n-1},y_{n-1}))$ be an ordered tuple of pairs in $\spt \pi$.
	We say $((x_2,y_2),\allowbreak \ldots,\allowbreak(x_{n-1},y_{n-1}))$ is a \emph{short-cut} for $(x_1,y_n) \in X \times Y$ if $(x_{i},y_{i+1}) \in \neigh$ for $i=1,\ldots,n-1$ and
	\begin{equation}
		\label{eq:ShortCutCondition}
		c(x_1,y_n) \geq c(x_1,y_2) + \sum_{i=2}^{n-1}
			\left[c(x_{i},y_{i+1})-c(x_i,y_i) \right]\,.
	\end{equation}
\end{definition}

\begin{proposition}
	\thlabel{thm:ShortCutImpliesConstraint}
	For a set $\neigh \subset X \times Y$ let $(\pi,(\alpha,\beta))$ be a pair of local primal and dual optimizers. Assume for a pair $(x_1,y_n) \notin \neigh$ there exists a short-cut within $\neigh$. Then the dual constraint \eqref{eq:OTDualConstraints} corresponding to $(x_1,y_n)$ is satisfied.
\end{proposition}
\begin{proof}
	Let $((x_2,y_2),\ldots,(x_{n-1},y_{n-1}))$ be a short-cut. From \eqref{eq:OTDualConstraints} restricted  to $\neigh$ and \eqref{eq:PDCondition} we find
	\begin{align*}
		\beta(y_{i+1}) & \leq c(x_i,y_{i+1}) - \alpha(x_i) & \text{for } & i=1,\ldots,n-1\,, \\
		- \beta(y_i) & = - c(x_i,y_i) + \alpha(x_i) & \text{for } & i=2,\ldots,n-1
	\end{align*}
	and by summing these up one gets
	\begin{align*}
		\alpha(x_1) + \beta(y_n) & \leq c(x_1,y_2) + \sum_{i=2}^{n-1} \left[ c(x_{i},y_{i+1}) - c(x_i,y_i) \right]\,.
	\end{align*}
	Validity of the dual constraint corresponding to $(x_1,y_n)$ follows from \eqref{eq:ShortCutCondition}. \qedhere
\end{proof}

So is there a clever way to choose a small set $\neigh$ such that there is a short-cut for every $(x,y) \notin \neigh$? But explicitly checking existence of short-cuts for each pair is far too expensive. 
In the next section introduce a simple sufficient condition for the existence of short-cuts.

\subsection{Shielding Condition}
The shielding condition is a local criterion to ensure existence of short-cuts. This eventually allows the verification of global optimality via local optimality. The shielding condition is illustrated in Fig.~\ref{fig:Shielding}.

\begin{figure}
	\centering%
	\includegraphics{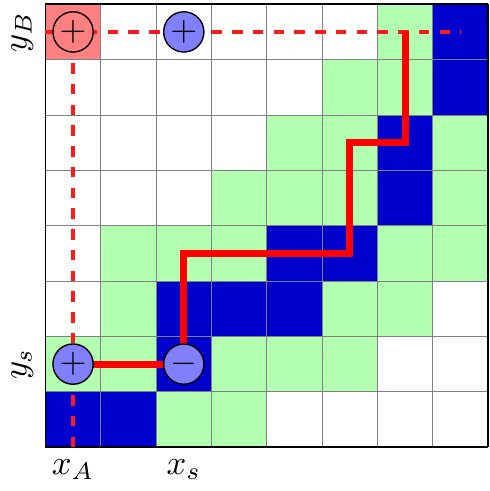}%
	\caption{Discrete transport problem: $\spt \pi$ (dark blue) and local neighbourhood $\neigh$ (green). When $(x_A,y_B)$ is shielded by $(x_s,y_s)$ (indicated by red and blue circles) then the constraint $(x_A,y_B)$ is implied by the constraints $(x_A,y_s)$, $(x_s,y_s)$ and $(x_s,y_B)$. Thus the problem of finding a short-cut for $(x_A,y_B)$ is reduced to the problem of finding a short-cut for $(x_s,y_B)$. If this can be repeated for all $x_s$, eventually the constraint $(x_A,y_B)$ is implied by constraints in $\neigh$ alone.}
	\label{fig:Shielding}
\end{figure}

\begin{definition}[Shielding Condition]
	\label{def:ShieldingCondition}
	Let $x_A \in X$, $(x,y) \in X \times Y$ and $y_B \in Y$. We say $(x_s,y_s)$ shields $x_A$ from $y_B$ when
	\begin{align}
		\label{eq:ShieldingCondition}
		c(x_A,y_B) - c(x_s,y_B) > c(x_A,y_s) - c(x_s,y_s)\,.
	\end{align}
\end{definition}	
The shielding condition states that $\{(x_A,y_s),(x_s,y_B)\}$ is (`strictly') c-cyclically monotone \cite[Chap.~5]{Villani-OptimalTransport-09}.
It implies that suitable $n$-tuples in $\spt \pi$ are in fact short-cuts.

\begin{proposition}
	\thlabel{thm:ShieldingImpliesShortcut}
	For a given coupling $\pi$ let $(x_1,y_n) \in X \times Y$ and $((x_2,y_2),\allowbreak \ldots, \allowbreak (x_{n-1},y_{n-1}))$ be an ordered tuple in $\spt \pi$. If $(x_{i},y_{i+1}) \allowbreak \in \allowbreak \neigh$ for $i=1,\ldots,n-1$ and $(x_{i+1},y_{i+1})$ shields $x_i$ from $y_n$ for $i=1,\ldots,n-2$ then the tuple is a short-cut for $(x_1,y_n)$.
\end{proposition}
\begin{proof}
	We need to show that \eqref{eq:ShortCutCondition} holds. For $i=1,\ldots,n-2$ we have from \eqref{eq:ShieldingCondition}
	\begin{align*}
		c(x_i,y_n) - c(x_{i+1},y_{n}) > c(x_{i},y_{i+1}) - c(x_{i+1},y_{i+1})\,.
	\end{align*}
	Summing up yields
		$c(x_1,y_n) >
			\sum_{i=1}^{n-2} \left[ c(x_{i},y_{i+1}) - c(x_{i+1},y_{i+1}) \right] + c(x_{n-1},y_{n})$. \qedhere
\end{proof}

We now introduce a sufficient condition for a set $\neigh$ such that short-cuts exist for all $(x,y) \notin \neigh$.
\begin{definition}[Shielding Neighbourhood]
	\label{def:ShieldingNeighbourhood}
	For a given coupling $\pi$ we say that a neighbourhood $\neigh \subset X \times Y$, $\neigh \supset \spt \pi$ is \emph{shielding} if for every pair $(x_A,y_B) \in X \times Y$ at least one of the following is true:
	\begin{enumerate}[label=(\roman*),ref=\roman*]
		\item $(x_A,y_B) \in \neigh$.
		\item There exists some $(x_s,y_s) \in \spt \pi$ with $(x_A,y_s) \in \neigh$ such that $(x_s,y_s)$ shields $x_A$ from $y_B$.
	\end{enumerate}
\end{definition}		

\begin{proposition}[Existence of Short-Cuts]
	\thlabel{thm:ShortCutExistence}
	For a given coupling $\pi$ let $\neigh$ be a shielding neighbourhood. Then there exists a short-cut for every $(x_A,y_B) \in (X \times Y) \setminus \neigh$.
\end{proposition}

For the proof we use the following auxiliary algorithm.

\begin{algorithm}
			\label{alg:PathConstruction}
			Input: $(\neigh,\pi)$ as specified in {\protect \thref{thm:ShortCutExistence}}, $(x_A,y_B)\allowbreak \in (X \times Y) \setminus \neigh$.
			Output: short-cut for $(x_A,y_B)$. \newline
			\tn{
			\indent \texttt{$n \leftarrow 1$; $x_1 \leftarrow x_A$} \newline
			\indent \texttt{while $(x_n,y_B) \notin \neigh$:} \newline
			\indent \indent \texttt{find $(x_{n+1},y_{n+1}) \in \spt \pi$ with $(x_{n},y_{n+1}) \in \neigh$ such that $\backslash$} \newline
			\indent \indent \indent \indent \texttt{$(x_{n+1},y_{n+1})$ shields $x_n$ from $y_B$} \newline
			\indent \indent \texttt{$n \leftarrow n+1$} \newline
			\indent \texttt{return $((x_2,y_2),\ldots,(x_{n},y_{n}))$}
			}
\end{algorithm}
\begin{remark}[Pseudo-Code Syntax]
	The pseudo-code syntax used in this article is based on \texttt{Python}, combined with mathematical expressions e.g.\ for set construction. $x \leftarrow y$ denotes value assignment to variables.
\end{remark}

\begin{proof}[Proof of {\protect \thref{thm:ShortCutExistence}}]
	We show that Algorithm \ref{alg:PathConstruction} always terminates and returns a valid short-cut for any pair $(x_A,y_B) \notin \neigh$.
	
	By virtue of Definition \ref{def:ShieldingNeighbourhood} in each iteration either the \texttt{while}-loop terminates or there exists a suitable pair $(x_{n+1},y_{n+1}) \in \spt \pi$ shielding $x_n$ from $y_B$.
	Since the number of elements in $\spt \pi$ is finite, either the loop eventually terminates or a cycle occurs.
	Assume we had found a cycle, i.e.~$(x_i,y_i) = (x_{k},y_{k})$ for some $i > 1$, $k\geq i+1$. By adding up the shielding condition \eqref{eq:ShieldingCondition} around the cycle we find:
	\begin{align*}
		\sum_{j=i}^{k-1} \left[ c(x_{j},y_{j+1}) - c(x_{j+1},y_{j+1}) \right] & <
		\sum_{j=i}^{k-1} \left[ c(x_{j},y_B) - c(x_{j+1},y_B) \right] \\
		& = c(x_i,y_B) - c(x_{k},y_B) = 0
	\end{align*}	
	Let $(\alpha,\beta)$ be any pair of corresponding local dual optimizers.
	The dual constraints \eqref{eq:OTDualConstraints} on $\neigh$ and the primal-dual relation \eqref{eq:PDCondition} imply:
	\begin{align*}
		0 & = \sum_{j=i}^{k-1} \left[ \alpha(x_{j}) - \alpha(x_{j+1}) \right]
			\leq \sum_{j=i}^{k-1} \left[
				\big( c(x_{j},y_{j+1})-\beta(y_{j+1}) \big) -
				\big( c(x_{j+1},y_{j+1})-\beta(y_{j+1}) \big)
				\right] \\
			& = \sum_{j=i}^{k-1} \left[
				c(x_{j},y_{j+1}) - c(x_{j+1},y_{j+1})
				\right]
	\end{align*}
	This contradiction implies that no cycles occur and the algorithm eventually terminates.
	
	Once the iteration is terminated, \thref{thm:ShieldingImpliesShortcut} provides that $((x_2,y_2),\ldots,(x_{n-1},y_{n-1}))$ is a short-cut for $(x_A,y_B)$.
\end{proof}
\begin{remark}
	\label{rem:ShieldingLEQ}
	Note that the strict inequality in \eqref{eq:ShieldingCondition} is merely required to guarantee termination of Algorithm \ref{alg:PathConstruction}. Proposition \ref{thm:ShieldingImpliesShortcut} already follows from $\geq$ in \eqref{eq:ShieldingCondition}.
\end{remark}
We emphasize that Algorithm \ref{alg:PathConstruction} is not intended to be run to find short-cuts. This would be immensely more expensive than directly checking the corresponding dual constraint. It is merely a tool for the proof.
We now summarize:
\begin{corollary}[Global Optimality from Local Optimality]
	\thlabel{thm:GlobalOptimality}
		Let $\pi$ be a local primal optimizer for a given feasible neighbourhood $\neigh$ and let $\neigh$ be shielding for $\pi$. Then $\pi$ is globally optimal.
\end{corollary}
\begin{proof}
	Let $(\alpha,\beta)$ be any local dual optimizers w.r.t.~$\neigh$.
	From \thref{thm:ShortCutExistence} follows existence of short-cuts for all dual constraints not covered by $\neigh$ and by \thref{thm:ShortCutImpliesConstraint} follows that $(\alpha,\beta)$ are globally dual feasible. Due to strong duality of linear programs, $\pi$ and $(\alpha,\beta)$ have vanishing duality gap, therefore $(\pi,(\alpha,\beta))$ are global primal and dual optimizers.
\end{proof}


\section{A Sparse Multi-Scale Algorithm}
\label{sec:Algorithm}
\subsection{Basic Algorithm}
	\label{sec:AlgorithmBase}
	\thref{thm:GlobalOptimality} can be used to construct an efficient sparse algorithm for large OT problems. The main ingredients of the algorithm are two maps:
	\begin{enumerate}[label=(\roman*),ref=\roman*]
		\item $\funcSolve: 2^{X \times Y} \rightarrow \Pi(\mu,\nu)$ such that for $\neigh \in 2^{X \times Y}$ the coupling $\funcSolve(\neigh)$ is locally primal optimal w.r.t.~$\neigh$.
			When $\neigh$ is sparse, any discrete OT solver can quickly provide an answer.
		\item $\funcShield: \Pi(\mu,\nu) \rightarrow 2^{X \times Y}$ such that for $\pi \in \Pi(\mu,\nu)$ the neighbourhood $\funcShield(\pi)$ is shielding for $\pi$.
			It is important for efficiency that $\funcShield(\pi)$ is sparse. To design such a map one must use the geometric structure of the cost function. In Sect.~\ref{sec:Shielding} we discuss how to implement $\funcShield$ for several costs.
	\end{enumerate}
	\thref{thm:GlobalOptimality} entails a `chicken and egg'-problem: For a given $\neigh_1$ let $\pi_2\allowbreak =\allowbreak \funcSolve(\neigh_1)$. But if $\pi_2$ is not globally optimal, then $\neigh_1$ cannot be shielding w.r.t.~$\pi_2$.
	Conversely, for some $\pi_2$ let $\neigh_2 = \funcShield(\pi_2)$, but $\pi_2$ will only be locally optimal w.r.t.~$\neigh_2$ if it is globally optimal.
	To find a configuration $(\neigh,\pi)$ such that both criteria are satisfied simultaneously, one can iterate both maps.

	\begin{algorithm}[\funcSolveShortCuts]
		\label{alg:ShortCutSolver}
		Input: initial feasible neighbourhood $\neigh_1$.
		Output: global optimizer $\pi$, neighbourhood $\neigh$ which is shielding for $\pi$.\newline
		\tn{
			\indent \texttt{$k \leftarrow 1$} \newline
			\indent \texttt{do:}\newline
			\indent \indent \texttt{
				$\pi_{k+1} \leftarrow \funcSolve(\neigh_k)$;
				$\neigh_{k+1} \leftarrow \funcShield(\pi_{k+1})$;
				$k \leftarrow k+1$} \newline
			\indent \texttt{until ($k>2$) and ($C(\pi_k) = C(\pi_{k-1})$)} \newline
			\indent \texttt{return $(\pi_k,\neigh_k)$}
		}
	\end{algorithm}
	\begin{proposition}
		For a feasible initial neighbourhood $\neigh_1$ Algorithm \ref{alg:ShortCutSolver} terminates after a finite number of iterations and returns a global primal optimizer.
	\end{proposition}
	\begin{proof}
		For $k>1$ have $\neigh_k = \funcShield(\pi_k)$ and $\pi_{k+1} = \funcSolve(\neigh_k)$. So $\spt \pi_k \subset \neigh_k$ and therefore $\pi_k$ is a feasible coupling when computing the optimal $\pi_{k+1}$, with support restricted to $\neigh_k$.
		It follows $C(\pi_{k+1}) \leq C(\pi_k)$. Since the problem is finite-dimensional one must eventually find $C(\pi_{k+1}) = C(\pi_k)$ and the algorithm terminates.
		Then $\pi_k$ is locally optimal w.r.t.~$\neigh_k$ and by construction $\neigh_k$ is shielding for $\pi_k$. Consequently $\pi_k$ and $\pi_{k+1}$ are globally optimal.
	\end{proof}
		
	We have rigorously established that Algorithm \ref{alg:ShortCutSolver} does terminate and return a global optimizer.
	In Sect.~\ref{sec:Numerics} we will demonstrate numerically that under two conditions it is in fact very efficient:
	\begin{enumerate}[label=(\roman*),ref=\roman*]
		\item As mentioned earlier, when $\neigh_k$ is sparse, calling the solver to compute the next coupling $\pi_{k+1} \leftarrow \funcSolve(\neigh_k)$ will be fast. In Sect.~\ref{sec:Shielding} we will show how to implement $\funcShield$ for several costs.
		\item When the initial guess $\neigh_1$ is good, only few iterations will be required. In Sect.~\ref{sec:AlgorithmHierarchy} we present a heuristic multi-scale scheme that works well in practice.
	\end{enumerate}
	
\subsection{Multi-Scale Scheme}
	\label{sec:AlgorithmHierarchy}
	
	The purpose of Algorithm \ref{alg:ShortCutSolver} is to accelerate the solving of large problems by starting from a smart initial guess for the sparse neighbourhood and then quickly solving a sequence of sparse problems until convergence instead of trying to solve the dense problem directly.
	As in \cite{MultiscaleTransport2011,SchmitzerSchnoerr-SSVM2013,ObermanOptimalTransportationLP2015} we approximate the original problem by multiple levels of successively coarser problems and then solve the original problem from coarse to fine. At each resolution we use the support of the optimal coupling as initialization $\neigh_1$ at the subsequent finer scale.
	
	It is verified empirically that this heuristic scheme works well in practice (see \cite{MultiscaleTransport2011,SchmitzerSchnoerr-SSVM2013,ObermanOptimalTransportationLP2015} and Section \ref{sec:Numerics}). Note however that \emph{we do not make any rigorous claims} on the computational efficiency of it. See Sect.~\ref{sec:Conclusion} for a discussion of the computational complexity.
	
	We now describe the multi-scale scheme in some more detail.
	\begin{definition}[Hierarchical Partition and Multi-Scale Measure Approximation \cite{SchmitzerSchnoerr-SSVM2013}]
		\label{def:HierarchicalPartition}
		For a discrete set $X$ a \emph{hierarchical partition} is an ordered tuple $(\hpartX_0,\ldots,\hpartX_K)$ of partitions of $X$ where $\hpartX_0 = \{ \{x\} \colon x \in X\}$ is the trivial partition of $X$ into singletons and each subsequent level is generated by merging cells from the previous level, i.e.~for $k \in \{1,\ldots,K\}$ and any $\hcellX \in \hpartX_k$ there exists some $\hat{\hpartX} \subset \hpartX_{k-1}$ such that $\hcellX = \bigcup_{\hat{\hcellX} \in \hat{\hpartX}} \hat{\hcellX}$.
		For simplicity we assume that the coarsest level is the trivial partition into one set: $\hpartX_K = \{X\}$.
		We call $K>0$ the \emph{depth} of $\hpartX$.
		
		This implies a directed tree graph with vertex set $\bigcup_{k'=0}^K \hpartX_{k'}$ and for $k \in \{1,\ldots,K\}$ we say $\hcellX' \in \hpartX_{j}$, $j<k$, is a \emph{descendant} of $\hcellX \in \hpartX_k$ when $\hcellX' \subset \hcellX$. We call $\hcellX'$ a \emph{child} of $\hcellX$ for $j=k-1$, and a \emph{leaf} for $j=0$.
		
		For some $\mu \in \prob(X)$ its \emph{multi-scale measure approximation} is the tuple $(\mu_0,\ldots,\mu_K)$ of probability measures $\mu_k \in \prob(\hpartX_k)$ defined by $\mu_k(\hat{\hpartX}) = \mu(\bigcup_{\hcellX \in \hat{\hpartX}} \hcellX)$ for all subsets $\hat{\hpartX} \subset \hpartX_k$ and $k=0,\ldots K$.
		
		For convenience we often identify $X$ with the finest partition level $\hpartX_0$, the set of singletons, and $\mu$ with $\mu_0$.
	\end{definition}
	In the examples discussed in this article $X$ and $Y$ are point clouds in $\R^n$ and the cost $c$ is originally defined on the full continuous space $\R^n \times \R^n$. We use hierarchical $2^n$-trees as partitions (some minor adaptions are necessary for the sphere discussed in Sect.~\ref{sec:ShieldingSphere}).
	\begin{definition}[Hierarchical $2^n$-trees and Hierarchical Costs]
	\label{def:2NTrees}
	For a given finite point cloud $X \subset \R^n$ and a desired depth $K>0$ we first choose an axis-aligned hypercube $\mc{Q}_K \subset \R^n$ that contains $X$. Correspondingly we set the coarsest partition layer to $\hpartX_K = \{X\}$.
	Then we divide $\mc{Q}_K$ into $2^n$ equal-sized smaller cubes $\{\mc{Q}_{K-1,i}\}_i$ parallel to the axes and set the corresponding partition layer to $\hpartX_{K-1}=\{X \cap \mc{Q}_{K-1,i}\}_i$. Empty cubes may be ignored. We repeat this recursively until level $1$ is reached. Then we add the layer of singletons $\hpartX_0 = \{ \{x\} \colon x \in X\}$.
	Clearly this produces a valid hierarchical partition $(\hpartX_0,\ldots,\hpartX_K)$ of $X$.
	
	Moreover, for each cube $\mc{Q}_{k,i}$ and the corresponding cell $\hcellX_i \in \hpartX_k$ at some partition level $k>0$ we define the \emph{representative} $\rep(\hcellX_i)$ as the center of $\mc{Q}_{k,i}$ and the \emph{radius} $\rad(\hcellX_i)$ as half of the diameter of $\mc{Q}_{k,i}$ such that
	\begin{align}
		\tn{for all} \quad x \in \mc{Q}_{k,i} \quad \tn{have} \quad |\rep(\hcellX_i)-x| \leq \rad(\hcellX_i)\,.
	\end{align}
	For $k=0$ we define for each $x \in X$, i.e.~$\{x\} \in \hpartX_0$, $\rep(\{x\})=x$ and $\rad(\{x\})=0$.
	
	We can now use the representatives to define a hierarchical cost function: for a cost function $c: \R^n \times \R^n \rightarrow \Rinf$ and two hierarchical $2^n$- trees $(\hpartX_0,\ldots,\hpartX_K)$ and $(\hpartY_0,\ldots,\hpartY_K)$ let
	\begin{align}
		c_k : \hpartX_k \times \hpartY_k \rightarrow \Rinf, \qquad
		c_k(\hcellX,\hcellY) = c(\rep(\hcellX),\rep(\hcellY))\,.
	\end{align}
	The radius will be useful for efficient construction of shielding neighbourhoods in Sect.~\ref{sec:Shielding}.
	\end{definition}
	Now the necessary ingredients are prepared to formally define the multi-scale variant of a discrete optimal transport problem.
	\begin{definition}[Multi-Scale Representation of Optimal Transport Problem]
		Let $X$, $Y$ be finite sets and $\mu \in \prob(X)$, $\nu \in \prob(Y)$.
		Let $(\hpartX_0,\ldots,\hpartX_K)$ and $(\hpartY_0,\ldots,\hpartY_K)$ be hierarchical partitions of $X$ and $Y$ with equal depth $K$,
		let $(\mu_0,\ldots,\mu_K)$ and $(\nu_0,\ldots,\nu_K)$ be multi-scale measure approximations of $\mu$ and $\nu$ over the hierarchical partitions
		and let $(c_0,\ldots,c_K)$ be a tuple of cost functions $c_k : \hpartX_k \times \hpartY_k \rightarrow \Rinf$, $k=0,\ldots,K$.
		
		Then for $k=0,\ldots,K$ we refer to the optimization problems \eqref{eq:OT} and \eqref{eq:OTDual} for the sets $\hpartX_k$, $\hpartY_k$, the marginals $\mu_k$, $\nu_k$ and the cost $c_k$ as the approximate problems at scale $k$.
		
		Note that the problem at scale $k=0$ is identical to the original problem.
	\end{definition}
	\noindent Finally, we describe how the sparse iterative Algorithm \ref{alg:ShortCutSolver} is combined with the multi-scale scheme.
	\begin{algorithm}[\funcSolveMultiScale]
		\label{alg:MultiScale}
		Input: multi-scale OT problem. Output: global optimizer of original problem.
		Notes: $\tn{\texttt{solveDense}}(k)$ refers to solving the dense problem at scale $k$,
			$\funcSolveShortCuts(k,\neigh_1)$ refers to calling Algorithm \ref{alg:ShortCutSolver} at scale $k$ with initial neighbourhood $\neigh_1$. (We are only interested in the optimal coupling $\pi$ and neglect the final shielding neighbourhood, which is also returned by $\funcSolveShortCuts$.) \newline
		\tn{
		\indent \texttt{$k \leftarrow K$} \newline
		\indent \texttt{$\pi \leftarrow \text{solveDense}(k)$} \newline
		\indent \texttt{while $k>0$:}\newline
		\indent \indent \texttt{$k \leftarrow k-1$} \newline
		\indent \indent \texttt{$\neigh_1 \leftarrow \{\}$} \newline
		\indent \indent \texttt{for $(x,y) \in \spt \pi$:} \newline
		\indent \indent \indent \texttt{$\neigh_1 \leftarrow \neigh_1 \cup (
			\text{children}(x) \times \text{children}(y))$} \newline
		\indent \indent \texttt{$\pi \leftarrow \funcSolveShortCuts(k,\neigh_1)$} \newline
		\indent \texttt{return $\pi$}
		}
	\end{algorithm}
	\begin{remark}[Validity of Algorithm \ref{alg:MultiScale}]
		For finite costs $c < \infty$ it is easy to see that the initial $\neigh_1$ constructed by Algorithm \ref{alg:MultiScale} are feasible and global optimality of the final coupling $\pi$ is inherited from Algorithm \ref{alg:ShortCutSolver}.
		
		If the cost can be infinite, there is no trivial relation between feasibility on different scales, as it depends on how mass will be distributed within the refined cells. In practice, potential infeasibility can be detected by adding `overflow' bins with a sufficiently high finite cost.
	\end{remark}
	
	\begin{remark}[Choice of Hierarchical Cost Function]
		\label{rem:HierarchicalCost}
		The choice how to define the hierarchical cost at coarser scales depends on the used algorithm. In \cite{SchmitzerSchnoerr-SSVM2013} the cost of a cell was defined recursively by taking the minimum over the cost of its children. In this way, dual feasibility on the coarse scale implied dual feasibility on the refined scale, enabling the efficient localization of violated constraints.
		
		Here we follow a different approach:
		In Definition \ref{def:2NTrees} we assigned a representative to each partition cell and defined the hierarchical cost by evaluating an underlying continuous cost at the representatives.
		This is important for solving the problems at coarser levels in Algorithm \ref{alg:MultiScale}.
		As will be discussed in Sect.~\ref{sec:Shielding}, the geometry of the problem is a key ingredient to constructing shielding neighbourhoods. Hence, we must make sure that the coarser problems still look like transport problems on the same underlying continuous space.		
		With the given setup, when solving the problem at scale $k>0$, we can just forget about all levels $0 \leq i < k$ and pretend the finest level is given by the point clouds $\{\rep(\hcellX) : \hcellX \in \hpartX_k \}$ and $\{\rep(\hcellY) : \hcellY \in \hpartY_k \}$.
		
		Hence, from now on we can always assume to solve the finest layer.
	\end{remark}


\section{Constructing Shielding Neighbourhoods}
	\label{sec:Shielding}
	The concepts of Section \ref{sec:ShortCuts} have been formulated for general cost functions. For the execution of \funcSolveShortCuts, Algorithm \ref{alg:ShortCutSolver}, we need a function $\funcShield$ that efficiently generates a sparse shielding neighbourhood for a given coupling.
	This is where the particular geometric structure of the cost function must be exploited.
	
	In this Section we discuss how $\funcShield$ can be designed for different types of ground costs. We start by describing the general outline in Sect.~\ref{sec:ShieldingGeneral}.
	The particularly important case of the squared Euclidean distance on $\R^n$ is treated in Sect.~\ref{sec:ShieldingSqrEuclidean} and the concept of short-cuts is applied to the continuous setting.
	In Sect.~\ref{sec:ShieldingConvex} we consider more general strictly convex functions on $\R^n$ and provide explicit formulas for the $p$-th power of the Euclidean distance.
	To demonstrate the generality of the concept, in Sect.~\ref{sec:ShieldingSphere}, we look at the squared geodesic distance on the sphere.
	The extension to noisy and distorted variants of the above costs is discussed in Sect.~\ref{sec:ShieldingNoise}.

	For the squared Euclidean distance we give a rigorous bound on the cardinality of the constructed neighbourhood $\neigh$ (\thref{thm:ShieldingYHatBound}), showing that it will indeed be sparse. For the other cases we provide analogous intuitive arguments.
	
	Different cost functions require different approximation techniques. Hence, this Section cannot be exhaustive and cover all possible costs. Instead we cover several important cases. By describing the underlying ideas and strategy we hope to enable the reader to transfer the results to other suitable cost functions.

\subsection{General Considerations}
	\label{sec:ShieldingGeneral}
	\paragraph{The Basic Algorithm.} We now describe the general strategy for constructing shielding neighbourhoods for a given coupling $\pi$. Let us first give an informal description.
	We start by setting $\neigh=\spt \pi$ (which is assumed to be sparse) and then try to add a small set of entries to make $\neigh$ shielding for $\pi$. The idea is to find for each $x_A \in X$ a suitably chosen small set of `shielding candidates' $\{(x_{s,i},y_{s,i})\}_i \subset \spt \pi$ such that `almost all' $y_B \in Y$ will be shielded from $x_A$ by one of the candidates. Then, per $x_A$, we only need to add a few elements to $\neigh$ and $\neigh$ will remain sparse.
	From the brief review of the continuous setting in Sect.~\ref{sec:ShortCutsIntuition} we conjecture that it is reasonable that the $X$-part of the shielding candidates should form a small geometric neighbourhood of $x_A$ in $X$ to mimic the sets $U_i$. So for each $x_A \in X$ we fix a `shielding candidate set' $\shield(x_A) \subset X$ which contains a small `discrete neighbourhood' around $x_A$. We will briefly comment on this choice for each cost in the subsequent sections.
	
	Moreover, for any $x_s \in S(x_A)$ we need to find an element $y_s \in Y$ such that $(x_s,y_s) \in \spt \pi$. To achieve this, from a given coupling $\pi$ we will extract a map $t : X \rightarrow Y$ such that $(x,t(x)) \in \spt \pi$ for all $x \in X$. Of course, in a discrete setting such a map $t$ need neither be injective nor surjective. However, this is not required for the functionality of the algorithm. Note that the extraction of the map can be performed efficiently, even for large problems, when $\pi$ is stored in a suitable sparse data structure.
	
	After having established the sets $\shield(x_A)$ and the map $t$ the algorithm can be stated more formally:
	\begin{algorithm}[\funcShield]
		\label{alg:ConstructShielding}
		Input: a coupling $\pi$.
		Output: a shielding neighbourhood $\neigh$.\newline
		\tn{
		\indent \texttt{$\neigh \leftarrow \spt \pi$} \newline
		\indent \texttt{from $\pi$ extract map $t : X \rightarrow Y$ such that $(x,t(x)) \in \spt \pi$ for all $x \in X$}
			\tn{ \hfill (\customlabel{item:ShieldAlgExtractT}{step-i}}\!\!) \newline
		\indent \texttt{for $x_A \in X$:}\newline
		\indent \indent \texttt{$\neigh \leftarrow \neigh \cup \{(x_A,t(x_s)) \colon
			x_s \in \shield(x_A) \}$}
		 \tn{ \hfill (\customlabel{item:ShieldAlgYNeigh}{step-ii}}\!\!) \newline
		\indent \indent \texttt{$\hat{Y} \leftarrow \{y_B \in Y : 
			y_B \text{ is not shielded from } x_A \text{ by } \backslash$}\newline
			\indent \indent \indent \indent \texttt{$(x_s,t(x_s)) \text{ for any } x_s \in S(x_A)
			\}$} \tn{ \hfill (\customlabel{item:ShieldAlgYLoop}{step-iii}}\!\!) \newline
		\indent \indent \texttt{$\neigh \leftarrow \neigh \cup \{(x_A,y_B) \colon y_B \in \hat{Y}\}$} \newline
		\indent \texttt{return $\neigh$}
		}
	\end{algorithm}
	It is easy to verify that for a given coupling $\pi$ this algorithm does indeed produce a valid shielding neighbourhood $\neigh$.
	
	\paragraph{Determining $\hat{Y}$.} Since (\ref{item:ShieldAlgYLoop}) is within the loop over $x_A \in X$ it would be inefficient to naively iterate over all $y_B \in Y$ and having to check the shielding condition explicitly for each pair $(x_A,y_B) \in X \times Y$ and shielding candidate $(x_s,t(x_s))$.
	Instead we will determine the set $\hat{Y}$ in a hierarchical way, making use of the hierarchical structure introduced in Sect.~\ref{sec:AlgorithmHierarchy}.
	Let $(x_A,x_s,y_s) \in X \times X \times Y$ be a triplet with $x_s \in \shield(X)$, $(x_s,y_s) \in \spt \pi$. Introduce the family of functions
	\begin{align}
		\label{eq:Psi}
		\psi_{(x_1,x_2)}(y) & = c(x_1,y) - c(x_2,y)\,. \\
		\intertext{Then $x_A$ is shielded from $y_B$ by $(x_s,y_s)$ precisely if (see \eqref{eq:ShieldingCondition})}
		\label{eq:ShieldingConditionPsi}
		\psi_{(x_A,x_s)}(y_B) & > \psi_{(x_A,x_s)}(y_s)\,.
	\end{align}
	Let now $(\hpartY_0,\ldots,\hpartY_K)$ be a hierarchical partition over $Y$. Assume we are given a hierarchical lower bound $\hat{\psi}_{(x_A,x_s)}$ of $\psi_{(x_A,x_s)}$ on $\bigcup_{k=0}^K \hpartY_k$ as follows:
	\begin{subequations}
	\begin{align}
		\hat{\psi}_{(x_A,x_s)}(\{y\}) & \eqdef \psi_{(x_A,x_s)}(y) & \qquad
			\tn{for} & \qquad y \in Y,\,\tn{i.e.~} \{y\} \in \hpartY_0\,, \\
		\label{eq:PsiHat}
		\hat{\psi}_{(x_A,x_s)}(\hcellY) & \leq \psi_{(x_A,x_s)}(y) & \qquad
			\tn{for} & \qquad \hcellY \in \hpartY_k,\,k>0,\,y \in \hcellY\,.
	\end{align}
	\end{subequations}
	So when
	\begin{align}
		\label{eq:ShieldingConditionPsiHat}
		\hat{\psi}_{(x_A,x_s)}(\hcellY) > \psi_{(x_A,x_s)}(y_s)
	\end{align}
	then all leaves of $\hcellY$ are shielded from $x_A$ by $(x_s,y_s)$.
	Consequently one can determine $\hat{Y}$ by doing a coarse-to-fine check of \eqref{eq:ShieldingConditionPsiHat} over the hierarchical partition of $Y$. We start at the coarse nodes and if the check fails recursively perform the test at finer levels until eventually the test succeeds or some leaves must be added to $\hat{Y}$. This is more formally described by the next algorithm:
	
	\begin{algorithm}[\funcSearch]
		\label{alg:FindYHat}
		Input: $x_A \in X$, set of shielding candidates $\{(x_{s,i},y_{s,i})\}_i$, current level $k \in \{0,\ldots,K\}$, current root cell $\hcellY \in \hpartY_k$.
		Output: set of missed elements $\hat{Y}$ that are leaves of $\hcellY$.
		Notes: $\funcSearch(x_A,\{(x_{s,i},y_{s,i})\}_i,k,\hcellY)$ refers to recursively calling this algorithm again with different parameters.\newline
		\tn{
		\indent \texttt{$\hat{Y} \leftarrow \{\}$} \newline
		\indent \texttt{if ($\hat{\psi}_{(x_A,x_s)}(\hcellY) \leq \psi_{(x_A,x_s)}(y_s)$ for all $(x_s,y_s) \in \{(x_{s,i},y_{s,i})\}_i$):} \newline
		\indent \indent \texttt{if $k=0$:} \newline
		\indent \indent \indent \texttt{$\hat{Y} \leftarrow \hat{Y} \cup \hcellY$} \newline
		\indent \indent \texttt{if $k>0$:} \newline
		\indent \indent \indent \texttt{for $\hcellY' \in \text{children}(\hcellY)$:}\newline
		\indent \indent \indent \indent \texttt{$\hat{Y} \leftarrow \hat{Y} \cup
			\funcSearch(x_A,\{(x_{s,i},y_{s,i})\}_i,k-1,\hcellY')$} \newline
		\indent \texttt{return $\hat{Y}$}
		}
	\end{algorithm}
	Calling this algorithm with the initial parameters $(x_A,\{(x_s,t(x_s)) \,\colon\,x_s \in \shield(x_A)\},K,Y)$ then returns the set $\hat{Y}$ as desired (recall that we defined $\hpartY_K=\{Y\}$, Definition \ref{def:HierarchicalPartition}). We will numerically verify that this hierarchical search requires significantly less calls than a naive dense search (see Fig.~\ref{fig:ExpShieldingNShielding}).
	
	We are now prepared to discuss various types of cost functions in more detail. We elaborate on how to choose $\shield(\cdot)$ such that we expect sparse sets $\hat{Y}$ and give explicit formulas for the bound $\hat{\psi}_{(\cdot,\cdot)}$. In the following Sections we assume that $X$ and $Y$ are discrete point clouds in $\R^n$, and there is a cost defined on the full continuous space $c : \R^n \times \R^n \rightarrow \R$.
	As specified in Definition \ref{def:2NTrees} $(\hpartX_0,\ldots,\hpartX_K)$ and $(\hpartY_0,\ldots,\hpartY_K)$ are hierarchical $2^n$-trees over $X$ and $Y$ with representatives, radii and a corresponding hierarchical cost function. For Sect.~\ref{sec:ShieldingSphere} we need to make some minor adaptions to the Riemannian setting.
	
\subsection[Squared Euclidean Distance on R\textasciicircum n]{Squared Euclidean Distance on $\R^n$}
	\label{sec:ShieldingSqrEuclidean}
	The squared Euclidean distance is perhaps the most prominent cost for optimal transport. It also allows a particularly simple geometric interpretation of the shielding condition.
	
	In this section let $c(x,y) = |x-y|^2$.
	Then the shielding condition \eqref{eq:ShieldingCondition}, see also \eqref{eq:ShieldingConditionPsi}, for a triplet $x_A$, $(x_s,y_s)$ and $y_B$ is equivalent to
	\begin{align}
		\label{eq:ShieldingEuclideanHyperplane}
		\psi_{(x_A,x_s)}(y_B) - \psi_{(x_A,x_s)}(y_s) > 0
		\qquad \Leftrightarrow \qquad
		\la x_s-x_A, y_B-y_s \ra > 0\,.
	\end{align}
	Consider the hyperplane through $y_s$, normal to $x_s-x_A$. Then $(x_s,y_s)$ shields $x_A$ from all $y_B$ that lie on the side facing in direction $x_s-x_A$.
	For a given $x_A \in X$ the set of points $y \in \R^n$ for which \eqref{eq:ShieldingEuclideanHyperplane} is false for all $\{(x_s,y_s = t(x_s)) : x_s \in \shield(x_A)\}$ is given by the polytope $P$ with faces through $y_s$ and outward normals $x_s - x_A$. Then $\hat{Y} = P \cap Y$ (for the map $t$ and the set $\hat{Y}$ see Algorithm \ref{alg:ConstructShielding}, \ref{item:ShieldAlgExtractT} and \ref{item:ShieldAlgYLoop}). An illustration of this is given in Fig.~\ref{fig:ShieldingEuclidean}.

	\begin{figure}
		\centering
		\includegraphics{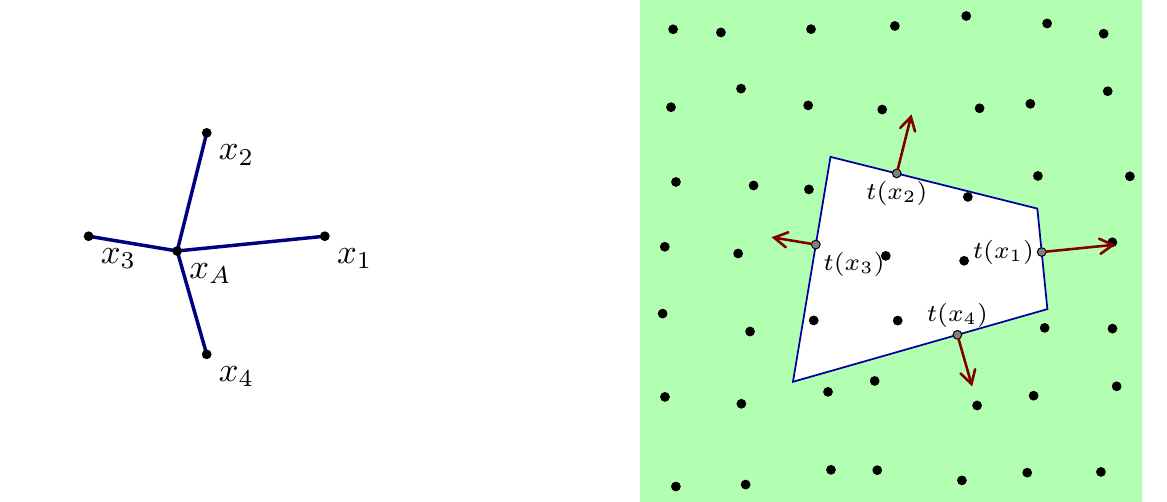}
		\caption{Illustrating the shielding condition for the squared Euclidean distance. \textbf{Left:} A point $x_A \in X$ with four points $\{x_1,x_2,x_3,x_4\} = \shield(x_A)$. \textbf{Right:} Point cloud $Y$ (black dots), the faces of the polytope $P$ (blue line) go through the points $t(x_i)$ with outward normals $x_i - x_A$, $i=1,\ldots,4$ (red arrows). $x_A$ is shielded from all points outside of $P$ (green area) by some $(x_i,t(x_i))$, $x_i \in \shield(x_A)$.}
		\label{fig:ShieldingEuclidean}
	\end{figure}

	The next proposition formalizes that under `plausible' regularity assumptions on $X$, $Y$ and the current coupling candidate $\pi$ the cardinality of $\hat{Y}$ is $\mc{O}(1)$ w.r.t.\ the cardinality of $Y$ and thus the cardinality of the generated shielding neighbourhood will be $\mc{O}(|X|)$ (as long as $|\shield(x_A)| = \mc{O}(1)$).
	\begin{proposition}[Bound on Cardinality of $\hat{Y}$]
		\thlabel{thm:ShieldingYHatBound}
		Let $X$ and $Y \subset \R^n$ be finite sets, $\pi$ a coupling on $\prob(X \times Y)$ and $\shield : X \rightarrow 2^X$ be an assignment of shielding candidate sets that satisfy:
		\begin{enumerate}[label=(\roman*)]
			\item There is a constant $q \in (0,1]$ such that for every $x_A \in X$ and any $v \in \R^n$ there is some $x_s \in \shield(x_A)$ such that
			\begin{align}
				\la v, x_s - x_A \ra \geq |v| \cdot |x_s-x_A| \cdot q\,.
			\end{align}
			$q$ can be interpreted as the cosine of the half of the maximal angle between any two $x_{1}-x_A$ and $x_{2}-x_A$ for $x_{1}$, $x_{2} \in \shield(x_A)$. \label{item:YHatBoundQ}
			\item There is a bound $0 < D < \infty$ such that for any $x_A \in X$ and any $x_s \in \shield(x_A)$ we have $|x_A-x_s| < D$. \label{item:YHatBoundD}
			\item There is a constant $0 < \rho < \infty$ such that for any ball $B_R \subset \R^n$ of radius $R>0$ we have $|Y \cap B_R| \leq \rho \cdot \vol_n(B_R)$. Where $\vol_n(B_R)$ gives the $n$-dimensional volume of the ball. The constant $\rho$ can be interpreted as an approximate upper bound on the point density in $Y$.
			\label{item:YHatBoundRho}
			\item The coupling $\pi$ is spatially regular in the sense that one can extract a map $t$ from $\pi$ (cf.~Algorithm \ref{alg:ConstructShielding}, \ref{item:ShieldAlgExtractT}) such that there is a constant $0 < L < \infty$ with $|t(x_1)-t(x_2)| \leq L \cdot |x_1-x_2|$ for all $x_1$, $x_2 \in X$. \label{item:YHatBoundPi}
		\end{enumerate}
		Then there is a constant $0 < C < \infty$ such that $\hat{Y} < C$ for all $x_A \in X$.
	\end{proposition}
	\begin{proof}
		Let $t$ be a map extracted from $\pi$ with Lipschitz constant $L< \infty$.
		For a given $x_A$ and any $y_B \in \R^n$ there is some $x_s \in \shield(x_A)$ such that
		\begin{align*}
			\la y_B-t(x_s), x_s-x_A \ra & = \la y_B - t(x_A) + t(x_A) - t(x_s), x_s - x_A \ra \\
			& \stackrel{\ref{item:YHatBoundQ},\ref{item:YHatBoundPi}}{\geq} |y_B-t(x_A)| \cdot |x_s-x_A| \cdot q - L\,|x_s-x_A|^2\\
			& \stackrel{\ref{item:YHatBoundD}}{\geq} (|y_B-t(x_A)| \cdot q - L \cdot D) \cdot |x_s-x_A|
		\end{align*}
		So for $|y_B-t(x_A)| > L \cdot D / q$ this is necessarily positive and thus the polytope $P$ of $y_P$ for which no shield exists must be contained in the closed ball $B(L \cdot D / q,t(x_A))$ of radius $L \cdot D / q$ around $t(x_A)$. Consequently, by \ref{item:YHatBoundRho}, $|\hat{Y}| \leq \rho \cdot \vol_n(B(L \cdot D / q,t(x_A)))$ which does not depend on $|Y|$ or $|X|$.
	\end{proof}
	\begin{remark}[Interpretation of \thref{thm:ShieldingYHatBound}]
		\label{rem:ShieldingYHatBoundInterpretation}
		Assumptions 	\ref{item:YHatBoundQ} to \ref{item:YHatBoundRho} depend only on $X$ and $Y$ and are rather `realistic'. They are met by all examples in this article (with some exceptions for $x_A$ at the `boundary' of $X$, see also Remark \ref{rem:ShieldingCartesianGrids}). Moreover assumptions \ref{item:YHatBoundQ} and \ref{item:YHatBoundD} provide useful guidance on how to choose $\shield(\cdot)$.
		
		For a given $\pi$ one can determine a suitable constant $L$ for \ref{item:YHatBoundPi} and thus in principle bound the size of $\neigh$ and consequently the complexity of the subsequent sparse problem. However, to estimate the full complexity of Algorithm \ref{alg:ShortCutSolver} one would need to fix $L$ in advance, hence an a priori estimate on the regularity of $\pi$ is required. This is considerably more difficult and therefore in this article we refrain from attempting to give a complete rigorous complexity analysis of the full multi-scale scheme involving Algorithms \ref{alg:ShortCutSolver} and \ref{alg:MultiScale}. See also Sect.~\ref{sec:Conclusion} for a discussion of the complexity.
	\end{remark}
	
	Independent of \thref{thm:ShieldingYHatBound} one can use the following function for the hierarchical search of $\hat{Y}$ outlined in the previous section (Algorithm \ref{alg:FindYHat}).
	\begin{proposition}
		\thlabel{thm:ShieldingPsiEuclidean}
		For a partition cell $\hcellY \in \hpartY_k$, $k=1,\ldots,K$ a hierarchical lower bound for $\psi_{(x_A,x_s)}$ is given by
	\begin{align}
		\hat{\psi}_{(x_A,x_s)}(\hcellY) & = \psi_{(x_A,x_s)}(\rep(\hcellY)) - 2\,|x_s-x_A|\,\rad(\hcellY)\,.
	\end{align}
	\end{proposition}
	\begin{proof}
	For some $y \in \hcellY$ check:
	\begin{align*}
		\psi_{(x_A,x_s)}(y) & = |x_A-y|^2 - |x_s-y|^2 = |x_A|^2 - |x_s|^2 - 2 \la y,x_A-x_s \ra \\
		& = |x_A|^2 - |x_s|^2 - 2 \la \rep(\hcellY),x_A-x_s \ra - 2 \la y-\rep(\hcellY),x_A-x_s \ra \\
		& = \psi_{(x_A,x_s)}(\rep(\hcellY)) - 2 \la y-\rep(\hcellY), x_A-x_s \ra \\
		& \geq \psi_{(x_A,x_s)}(\rep(\hcellY)) - 2\,|y-\rep(\hcellY)| \cdot |x_A-x_s| \geq \hat{\psi}_{(x_A,x_s)}(\hcellY)
	\end{align*}
	Therefore condition \eqref{eq:PsiHat} is satisfied.
	\end{proof}
	Intuitively, for the representative $\rep(\hcellY)$ one has to take into account an additional margin proportional to $\rad(\hcellY)$ to make sure all potential leaves of $\hcellY$ are on the right side of the hyperplane defined by \eqref{eq:ShieldingEuclideanHyperplane}.

	\begin{remark}[Particular Case: Cartesian Grids]
	\label{rem:ShieldingCartesianGrids}
	In many applications the discrete sets $X$ and $Y$ are not just random point clouds but lie on a Cartesian grid. Then $\hat{Y}$ can be determined directly via the grid structure without a hierarchical search.
	
	For simplicity assume for now $n=2$, higher dimensions work analogously. Assume $X$, $Y \subset \Z^2 \subset \R^2$ are regular orthogonal grids:
	\begin{align*}
		X = \{0, \ldots, N_{X,1} \} \times \{0, \ldots, N_{X,2}\}\,,\qquad
		Y = \{0, \ldots, N_{Y,1} \} \times \{0, \ldots, N_{Y,2}\}
	\end{align*}
	for some positive integers $N_{Z,i}$, $Z=X,Y$, $i=1,2$.
	For every $x_A \in X$ let $\shield(x_A)$ be the 4-neighbourhood of $x_A$ on the grid $X$ (potentially incomplete at boundaries and corners). Then for any $x_A$ with a complete neighbourhood one has
	\begin{align*}
		\{ x_s - x_A : x_s \in \shield(x_A) \} = \left\{
			\begin{pmatrix} 1 \\ 0 \end{pmatrix},
			\begin{pmatrix} 0 \\ 1 \end{pmatrix},
			\begin{pmatrix} -1 \\ 0 \end{pmatrix},
			\begin{pmatrix} 0 \\ -1 \end{pmatrix}
			\right\}
	\end{align*}
	and consequently the polytope $P$ for which condition \eqref{eq:ShieldingEuclideanHyperplane} is false for all $\{(x_s,y_s = t(x_s)) : x_s \in \shield(x_A)\}$ is a grid-aligned rectangle with sides going through the points $y_s$. This set can be accessed directly without a search by using the grid structure.
	
	Up to the points at the `boundary' of $X$ that do not have a full 4-neighbourhood, this set-up clearly satisfies assumptions \ref{item:YHatBoundQ} to \ref{item:YHatBoundRho} of \thref{thm:ShieldingYHatBound}.
	Note further that the subset of $X$ with incomplete 4-neighbourhoods is small compared to $X$ for large grids. Therefore the effect of the boundaries on the size of $\neigh$ is bounded. In practice, for a well-chosen initial coupling $\pi$ one usually has that mass from the boundary of $X$ is transported to the proximity of the corresponding boundary on $Y$, hence the unbounded side of $P$ does not have much overlap with $Y$. 
	\end{remark}	

	\begin{remark}[Application to the Continuous Problem]
		\label{rem:ShieldingEuclideanContinuous}
		We now return to the discussion in Sect.~\ref{sec:ShortCutsIntuition} and apply the concept of short-cuts to the continuous problem in $\R^n$. Given a transport map $T$, locally optimal on all patches $U_i \times V_i$, let the tuple of points $(y_1=y_A,\ldots,y_n=y_B)$, $y_i \in \spt \nu$ be taken from a straight line  between $y_A$ and $y_B$ in monotone order and sampled sufficiently fine such that every two successive points $y_i$, $y_{i+1}$ lie in the same patch $V_i$. Pick $x_i$ such that $T(x_i)=y_i$.
		
		Since $T$ is locally optimal on all $U_i \times V_i$ one has for $i=1,\ldots,n-2$
		\begin{align*}
			c(x_i,y_i) + c(x_{i+1},y_{i+1}) & \leq c(x_i,y_{i+1}) + c(x_{i+1},y_i) \\
			\Leftrightarrow \la x_{i+1}-x_i,y_{i+1}-y_i \ra & \geq 0 \\
			\Leftrightarrow \la x_{i+1}-x_i,y_{n}-y_{i+1} \ra & \geq 0 \qquad \tn{(since $y_n-y_{i+1}$ and $y_{i+1}-y_i$ are co-linear)} \\
			\Leftrightarrow 	c(x_i,y_{i+1}) + c(x_{i+1},y_{n}) & \leq c(x_i,y_{n}) + c(x_{i+1},y_{i+1})\,.
		\end{align*}
		It then follows that $(x_{i+1},y_{i+1})$ is shielding $x_i$ from $y_n$ for $i=1,\ldots,n-2$ (see \eqref{eq:ShieldingCondition}, \eqref{eq:ShieldingEuclideanHyperplane} and Remark \ref{rem:ShieldingLEQ}) and consequently the tuple $((x_2,y_2),\ldots,(x_{n-1},y_{n-1}))$ is a short-cut for $(x_1,y_n)$. Therefore the transport map $T$ is optimal.
		
		We see that for the squared Euclidean distance the shielding condition follows from local optimality along straight lines, which explains why local optimality is still sufficient in 1-dimensional discrete problems. In higher dimensions we cannot always jump along straight lines between grid points and even small deviations may break the shielding condition. This is why we must explicitly keep track of $\pi$ throughout Sections \ref{sec:ShortCuts} and \ref{sec:Shielding} and carefully choose the discrete equivalent of $U_i$ and $V_i$.
	\end{remark}

\subsection[Strictly Convex Functions on R\textasciicircum n]{Strictly Convex Functions on $\R^n$}
	\label{sec:ShieldingConvex}
	As discussed in Sect.~\ref{sec:IntroBackground}, Brenier's polar factorization theorem has been generalized to a large class of other convex costs. We now sketch how to construct shielding neighbourhood for strictly convex functions, discuss why we expect $\hat{Y}$ to be small and give an explicit formula for the hierarchical bound $\hat{\psi}_{(x_A,x_s)}$ (see \eqref{eq:PsiHat} and Algorithm \ref{alg:FindYHat}) for the $p$-th power of the Euclidean distance, $p \in (1,\infty)$.
	
	Throughout this section let $c(x,y) = h(x-y)$ for a strictly convex function $h : \R^n \rightarrow \R$. Then $\psi_{(x_A,x_s)}$, cf.~\eqref{eq:Psi}, is given by:
	\begin{align}
		\psi_{(x_A,x_s)}(y) & = h(x_A-y) - h(x_s-y)
	\end{align}
		This difference will not always have a simple closed form as for $c(x,y)=|x-y|^2$. Hence, we find a simpler approximate expression by means of the (strict) sub-gradient inequality:
	\begin{subequations}
	\label{eq:ShieldingConvexSubgradientBoth}
	\begin{align}
		\label{eq:ShieldingConvexSubgradient}
		\psi_{(x_A,x_s)}(y) & > \la \xi, x_A-x_s \ra \qquad \tn{with} \qquad
			\xi \in \partial h(x_s-y) \\
		\psi_{(x_A,x_s)}(y) & < \la \xi, x_A-x_s \ra \qquad \tn{with} \qquad
			\xi \in \partial h(x_A-y)
	\end{align}
	\end{subequations}
	A sufficient condition for $(x_s,y_s)$ shielding $x_A$ from $y_B$ is therefore:
	\begin{align}
		\label{eq:ShieldingConvexSufficient}
		\la \xi_1 - \xi_2, x_A-x_s \ra \geq 0 \qquad \tn{with} \qquad
			\xi_1 \in \partial h(x_s-y_B),
			\xi_2 \in \partial h(x_A-y_s)
	\end{align}
	
	\begin{remark}[Sparsity of $\hat{Y}$]
	We now give an informal analogue to \thref{thm:ShieldingYHatBound}. For every $x_A \in X$ let $\shield(x_A)$ be such that for any $v \in \R^n$ we can decompose 
	\begin{align}
		\label{eq:ShieldingConvexVDecomposition}
		v = |v|\,\sum_{x_s \in \shield(x_A)} \lambda(x_s) \cdot (x_A - x_s)
	\end{align}
	for non-negative coefficients $\lambda(x_s)$.
	For fixed $x_A$ and $y_B$ by monotonicity of the sub-differential we have
	\begin{align}
		\label{eq:ShieldingConvexPreCondition}
		\la \xi_1 - \xi_2, t(x_A)-y_B \ra & \geq 0 \qquad \tn{with } \qquad
		\xi_1 \in \partial h(x_A-y_B),
		\xi_2 \in \partial h(x_A-t(x_A)). \\
		\intertext{Decomposing $t(x_A)-y_B$ with \eqref{eq:ShieldingConvexVDecomposition} we obtain}
		\label{eq:ShieldingConvexPreConditionB}
		\la \xi_1 - \xi_2, x_A-x_s \ra & \geq 0
	\end{align}
	for some $x_s \in \shield(x_A)$.
	One could now assume some form of uniform convexity of $h$ (see for example \cite{ConvexFunctionalAnalysis-11}) and a lower bound on the maximal coefficient $\lambda(\cdot)$ to obtain a finite lower bound in inequality \eqref{eq:ShieldingConvexPreConditionB}, depending on $|t(x_A)-y_B|$ and conversely assume some local regularity of the sub-differential to bound the error inflicted by choosing $\xi_1 \in \partial h(x_s-y_B)$ and $\xi_2 \in \partial h(x_A-t(x_s))$ to get from (\ref{eq:ShieldingConvexPreCondition},\ref{eq:ShieldingConvexPreConditionB}) to \eqref{eq:ShieldingConvexSufficient}.
	However, in view of Remark \ref{rem:ShieldingYHatBoundInterpretation} this would have little further consequences. Therefore we content ourselves here with the intuitive argument that in practice $\hat{Y}$ will be small when $h$ is `reasonably regular and convex', which will later be confirmed numerically in Sect.~\ref{sec:NumericsPEucl}. Additional insight can be gained from Fig.~\ref{fig:ShieldingPEuclIllustration}.
	\end{remark}
	
	To actually construct shielding neighbourhoods one now needs to find a lower bound for \eqref{eq:ShieldingConvexSubgradient} on the hypercubes of the hierarchical $2^n$-tree.
	We give an explicit formula for $p$-th powers of the Euclidean distance, $p \in (1,\infty)$, but in principle this can be extended to other cost functions covered in \cite{McCannGangboOTGeometry1996}.
	\begin{proposition}
		\thlabel{thm:ShieldingPsiPEucl}
		Let $h(x-y)=|x-y|^p$, $p\in (1,\infty)$. For a partition cell $\hcellY \in \hpartY_k$, $k=1,\ldots,K$ a hierarchical lower bound for $\psi_{(x_A,x_s)}$ is given by
		\begin{align}
			\label{eq:ShieldingConvexPEuclSummary}
			\hat{\psi}_{(x_A,x_s)}(\hcellY) & = p\,R^{p-1} \cdot |x_A-x_s| \cdot \cos(\varphi) \\
			\intertext{where}
			\varphi & = \min\{ \pi, \measuredangle(x_A-x_s,x_s-\rep(\hcellY)) + \theta \} \\
			\theta & = \begin{cases}
				\arcsin(\rad(\hcellY)/|x_s-\rep(\hcellY)|) & \tn{for } \rad(\hcellY)<|x_s-\rep(\hcellY)| \\
				\pi & \tn{else}
				\end{cases} \\
			R & = \begin{cases}
				\max\{0, |x_s-\rep(\hcellY)|-\rad(\hcellY)\} & \tn{for } \cos(\varphi)\geq 0 \\
				|x_s-\rep(\hcellY)|+\rad(\hcellY) & \tn{else.}
				\end{cases}
		\end{align}
	\end{proposition}
	Here $\measuredangle(a,b)=\arccos(\la a,b \ra/(|a| \cdot |b|))$ denotes the (unsigned) angle between the two vectors $a$, $b \in \R^n$. This formula is not very handy for further analytic manipulation but it can readily be implemented for usage in Algorithm \ref{alg:FindYHat}. A proof of \thref{thm:ShieldingPsiPEucl} is given in Appendix \ref{apx:Proofs}.

	\begin{figure}
		\centering %
		\includegraphics[]{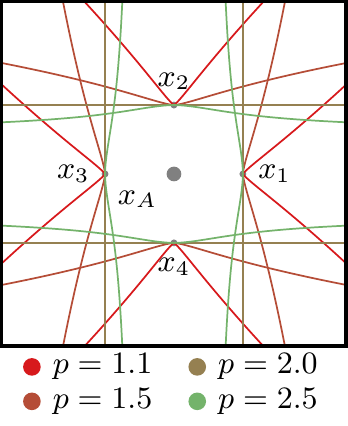} %
		\caption{Illustration of the shielding condition for $c(x,y)=|x-y|^p$ for various $p$ for a simple geometric set-up. %
			For $x_A=(0,0)$ and $(x_1,x_2,x_3,x_4)= (y_1,y_2,y_3,y_4) = ((1,0),(0,1),(-1,0),(0,-1))$ the coloured lines show for which $y_B$ the pair $(x_A,y_B)$ is shielded by one of the $(x_i,y_i)$, $i=1,\ldots,4$. For $p=2$ the boundaries are given by straight lines (cf.~Sect.~\ref{sec:ShieldingSqrEuclidean}), for $p>2$ the curves are bent `inwards', for $p<2$ `outwards'. As $p \searrow 1$ the area of shielded elements becomes smaller. This could be remedied by adding additional shielding candidates at intermediate angles.	
			For more general geometric set-ups the situation is more complicated, but the general behaviour remains.
		} %
		\label{fig:ShieldingPEuclIllustration} %
	\end{figure}

\subsection{Squared Geodesic Distance on Sphere}
	\label{sec:ShieldingSphere}
	There is also an extension of the polar factorization to compact Riemannian manifolds (see Sect.~\ref{sec:IntroBackground}). To demonstrate the flexibility of our framework, we now describe how to construct shielding neighbourhoods on the 2-dimensional unit sphere for the squared geodesic distance as cost.
	
	Let $\sphere$ be the 2-sphere in $\R^3$, denote by $d : \sphere \times \sphere \rightarrow \R_+$ the Riemannian geodesic distance on $\sphere$ and let $X$, $Y \subset \sphere$ be finite subsets of the sphere. 
	
	We need to slightly adapt the $2^n$-tree scheme for $\R^n$ as introduced in Def.~\ref{def:2NTrees}. Start by generating a hierarchical $2^3$-tree (octree) $(\hpartX_0,\ldots,\hpartX_K)$ over $X \subset \sphere \subset \R^3$ with representatives at the cube-centers as before. Then for each partition cell $\hcellX$ its representative $\rep(\hcellX)\in \R^3$ is projected onto $\sphere$ by normalizing its length $\rep(\hcellX) \leftarrow \rep(\hcellX)/|\rep(\hcellX)|$. Moreover, assign to each representative a metric radius $\rad(\hcellX)$ such that $d(\rep(\hcellX),x) \leq \rad(\hcellX)$ for all $x \in \hcellX$.
	Likewise, construct a hierarchical partition $(\hpartY_0,\ldots,\hpartY_K)$ with projected representatives and metric radii over $Y$.
	
	Then as before define
	\begin{align}
		c_k : \hpartX_k \times \hpartY_k \rightarrow \R\,, \qquad
		c_k(\hcellX,\hcellY) = d(\rep(\hcellX),\rep(\hcellY)))^2\,.
	\end{align}
	
	\begin{remark}[Sparsity of $\hat{Y}$]
	Let us again make an informal argument for why we expect that there exists a sparse shielding neighbourhood. For fixed $x_A \in X$ and $y_B \in Y$ consider the following function:
	\begin{align}
		F & : X \times Y \rightarrow \R, & (x_s,y_s) & \mapsto [c(x_A,y_B)-c(x_s,y_B)]-[c(x_A,y_s)-c(x_s,y_s)]
	\end{align}
	The pair $(x_s,y_s)$ shields $x_A$ from $y_B$ iff $F(x_s,y_s)>0$ (see \eqref{eq:ShieldingCondition}).
	Note that $F(x_A,t(x_A))=0$. Now try to find some $(x_s,t(x_s))$ with $F(x_s,t(x_s))>0$. We assume that $x_s$ is close to $x_A$ and by regularity of $t$ that $t(x_s)$ is close to $t(x_A)$. Consequently we do a first order expansion of $F$ in the tangent spaces of $x_A$ and $t(x_A)$. Note that $\nabla_{y_s} F(x_A,y_s)|_{y_s=t(x_A)}=0$ so the first order variation w.r.t.\ $y_s$ vanishes at $y_s=t(x_A)$. For the gradient in the first argument we find:
	\begin{align}
		\nabla_{x_s} F(x_s,t(x_A))|_{x_s=x_A} & = 2 \left[ \log_{x_A}(y_B) - \log_{x_A}(t(x_A)) \right]
	\end{align}
	where $\log_{x_A}$ denotes the logarithmic map on $\sphere$ that assigns to a point $y \in \sphere$ the vector in the tangent space at $x_A$ that induces the geodesic which reaches $y$ at time $1$. Ignoring issues like the cut locus we find that if $\shield(x_A)$ approximates the tangent space at $x_A$ sufficiently such that we can choose some $x_s$ lying in the direction given by $\log_{x_A}(y_B) - \log_{x_A}(t(x_A))$, we have found a shielding pair $(x_s,t(x_s))$.
	\end{remark}
	
	After this intuitive argument we turn to the construction of a hierarchical bound for Algorithm \ref{alg:FindYHat}. We approximate $\psi_{(\cdot,\cdot)}$ again via the sub-gradient inequality (cf.~\eqref{eq:Psi} and \eqref{eq:ShieldingConvexSubgradientBoth}):
	\begin{align}
		\psi_{(x_A,x_s)}(y) & = d(x_A,y)^2-d(x_s,y)^2 \\
		\psi_{(x_A,x_s)}(y_B) & \geq \xi \cdot ( d(x_A,y_B) - d(x_s,y_B) ) \qquad \tn{for} \qquad \xi = 2\,d(x_s,y_B)\,.
	\end{align}
	We then find:
	\begin{proposition}
		\thlabel{thm:ShieldingPsiSphere}
		For a partition cell $\hcellY \in \hpartY_k$, $k=1,\ldots,K$ and two points $x_A$, $x_s \in X$ choose a coordinate system in $\R^3$ such that
		\begin{align*}
			x_A & = \begin{pmatrix} 0 \\ 0 \\ 1 \\ \end{pmatrix}, &
			x_s & = \begin{pmatrix} \sin \theta_s \\ 0 \\ \cos \theta_s \\ \end{pmatrix}, &
			\rep(\hcellY) & = \begin{pmatrix}
				\sin \theta_B \cdot \cos \varphi_B \\
				\sin \theta_B \cdot \sin \varphi_B \\
				\cos \theta_B \\ \end{pmatrix}
		\end{align*}
		for suitable $\theta_s$, $\theta_B \in [0,\pi]$, $\varphi_B \in (-\pi,\pi]$. Then a hierarchical lower bound for $\psi_{(x_A,x_s)}$ is given by
		\begin{align}
			\hat{\psi}_{(x_A,x_s)}(\hcellY) & = \xi \cdot \Delta d_{\tn{min}} \\
			\intertext{with}
			\Delta d_{\tn{min}} & = \theta_{B,\tn{min}} - \arccos(
				\sin \theta_s \cdot \sin \theta_{B,\tn{min}} \cdot \cos \varphi_{B,\tn{max}} +
				\cos \theta_s \cdot \cos \theta_{B,\tn{min}}
				) \\
			\varphi_{B,\tn{max}} & = \min\{ \pi, |\varphi_B| + \Delta \varphi \} \\
			\Delta \varphi & = \begin{cases} \arccos \sqrt{ \frac{ \cos^2 \rad(\hcellY) - \cos^ 2 \theta_B}{1-\cos^2 \theta_B}} & \tn{if } \cos^2 \rad(\hcellY) > \cos^2 \theta_B \\
			\pi & \tn{else}
				\end{cases} \\
			\theta_{B,\tn{min}} & = \max \{ 0, \theta_B - \rad(\hcellY) \} \\
			\xi & = 2\,d_{\ast} \\
			d_{\ast} & = \begin{cases}
				\max\{0,\theta_B - \rad(\hcellY) \} & \tn{if } \Delta d_{\tn{min}} > 0 \\
				\min\{ \pi, \theta_B + \rad(\hcellY) \} & \tn{else}
				\end{cases}
		\end{align}			
	\end{proposition}
	As in the previous Section, this function is not particularly suitable for further analytic manipulation but it can readily be implemented numerically. A proof is given in Appendix \ref{apx:Proofs}.

\subsection{Noisy Cost Functions}
	\label{sec:ShieldingNoise}
	As discussed in Sect.~\ref{sec:IntroBackground} OT solvers based on the Monge-Amp\`ere equation require a particular form of the ground cost and even small perturbations can make a PDE based solver inapplicable. This can be particularly frustrating when the perturbation is only local and thus will most likely not affect the global structure of the optimal coupling.
	Such `noisy' costs arise for example in imaging, when geometric information is complemented by local image properties (e.g.\ \cite{SchmitzerSchnoerr-WassersteinModes2014}).
	
	Since the shielding condition is an inequality and because the internal solver in Algorithm \ref{alg:ShortCutSolver} can be combinatorial and thus does not rely on the cost function structure, the framework presented in this article can to some extent be adapted to the presence of noise.
	
	Throughout this section let
	\begin{align}
		c(x,y) & = \cGeo(x,y) + \eta\,\cNoise(x,y) + \lambda\,\cLip(x,y)
	\end{align}
	where $\cGeo$ is any of the \emph{geometric} cost functions discussed in Sections \ref{sec:ShieldingSqrEuclidean} to \ref{sec:ShieldingSphere}, $\cNoise : X \times Y \rightarrow [0, 1]$ is bounded but otherwise arbitrary and $\cLip$ is 1-Lipschitz in the first argument, $|\cLip(x_1,y)-\cLip(x_2,y)| \leq d(x_1,x_2)$ for the appropriate metric $d$. Random local noise is modelled by $\cNoise$, $\cLip$ describes other cost contributions that may have long-range structure. The positive constants $\eta$ and $\lambda$ determine the relative strength of the components.
	
	The shielding condition for this cost is
	\begin{align}
		[c(x_A,y_B)-c(x_s,y_B)] - [c(x_A,y_s)-c(x_s,y_s)] & > 0 \\
		\Leftrightarrow \qquad \Delta \cGeo + \eta \cdot \Delta \cNoise + \lambda \cdot \Delta \cLip & >  0
	\end{align}	
	with
	\begin{align}
		\Delta c_{\chi} = [c_{\chi}(x_A,y_B)-c_{\chi}(x_s,y_B)] - [c_{\chi}(x_A,y_s)-c_{\chi}(x_s,y_s)]
	\end{align}
	for $\chi = \tn{geo}, \tn{n}, \tn{L}$. Using the assumptions on $\cNoise$ and $\cLip$ we find
	\begin{align}
		 |\Delta \cNoise| & \leq 2, & |\Delta \cLip| & \leq 2\,d(x_A,x_s)\,.
	\end{align}
	So a sufficient condition for shielding is
	\begin{align}
		\Delta \cGeo & > 2\,\eta + 2\,\lambda\,d(x_A,x_s)
	\end{align}
	which is the original shielding condition of $\cGeo$ with an additional but constant offset (if, as discussed above, we can bound $d(x_A,x_s)$).
	In the case of the squared Euclidean distance this means that the shielding hyperplanes are shifted outwards by an additional margin of $\eta + \lambda \cdot |x_A-x_s|$. \thref{thm:ShieldingYHatBound} can then be adapted appropriately and one can still bound the size of $\hat{Y}$. Moreover, it is straight-forward to add this margin to the hierarchical bound $\hat{\psi}_{(x_A,x_s)}$, \thref{thm:ShieldingPsiEuclidean} and the Cartesian case, Remark \ref{rem:ShieldingCartesianGrids}.
	
	Recalling the informal discussions in Sections \ref{sec:ShieldingConvex} and \ref{sec:ShieldingSphere} this margin can in principle also be added to the bounds $\hat{\psi}_{(x_A,x_s)}$ for those costs, but for the remainder of this article we will only consider noisy variants of the squared Euclidean distance.


\section{Numerical Experiments}
	\label{sec:Numerics}
	We now present a series of numerical experiments to demonstrate %
	compatibility of the presented algorithm with current professional discrete solver software, %
	efficiency of the multi-scale scheme %
	and applicability of the scheme to practical problems.
	The code used for the numerical experiments is available from the author's website.\footnote{\url{https://www.ceremade.dauphine.fr/~schmitzer/}}
	
	\subsection{Implementation Details}
		We use three different algorithms as internal solvers $\funcSolve$ in Algorithm \ref{alg:ShortCutSolver}: %
		the network simplex \cite{NetworkFlows1993} implementation of CPLEX \cite{CPLEX}, and the network simplex and cost scaling \cite{GoldbergCostScaling1990} implementations of the LEMON library \cite{LEMON}. In the following we refer to these algorithms by the short-hands \texttt{CPLEX}, \texttt{LEMON-NS} and \texttt{LEMON-CS}.
		
		We measure and compare run-time, number of problem variables and other characteristics of solving test problems in different ways:
		\begin{itemize}
			\item \texttt{dense}: naive direct solving of the full dense problem,
		\end{itemize}
		and with Algorithms \ref{alg:ShortCutSolver} and \ref{alg:MultiScale} combined with the various methods for $\funcShield$ to construct shielding neighbourhoods developed in Sect.~\ref{sec:Shielding}:
		\begin{itemize}
			\item \texttt{grid}: directly using the grid structure, Remark \ref{rem:ShieldingCartesianGrids},
			\item \texttt{tree-2}: hierarchical search for squared Euclidean distance, \thref{thm:ShieldingPsiEuclidean},
			\item \texttt{tree-p}: hierarchical search for $p$-th power of Euclidean distance, \thref{thm:ShieldingPsiPEucl},
			\item \texttt{tree-sphere}: hierarchical search for geodesic distance on sphere, \thref{thm:ShieldingPsiSphere},
			\item \texttt{tree-noise}: hierarchical search for noisy squared Euclidean distance with noise slack, Sect.~\ref{sec:ShieldingNoise}.
		\end{itemize}
		
		\begin{remark}[Adaptive Re-Initialization]
			\label{rem:ReInitialization}
			In principle we can use the algorithms mentioned above as mere black boxes, solving each problem $\pi_{k+1} \leftarrow \funcSolve(\neigh_k)$ in Algorithm \ref{alg:ShortCutSolver} from scratch.
			
			However, when $\pi_k$ is close to optimality we expect $\neigh_k$ to only change slightly between successive problems. Hence, we try to preserve solver information between iterations to accelerate subsequent solving of similar problems.
			For the CPLEX network optimizer this can be achieved by using the functions \texttt{CPXNETgetbase()} and \texttt{CPXNETcopybase()} which extract and set a network basis for a given problem.
			
			The user interface of the LEMON library does not provide similar functions. Instead we try a simple trick: let $(\alpha_k,\beta_k)$ be the local dual optimizers w.r.t.\ $\neigh_{k-1}$ (which are provided by the LEMON Algorithms). Then we create a dual feasible pair $(\alpha_{k+1,\tn{init}},\beta_{k+1,\tn{init}})$ by first setting $(\alpha_{k+1,\tn{init}},\beta_{k+1,\tn{init}}) \leftarrow (\alpha_k,\beta_k)$ and then reducing $\alpha_{k+1,\tn{init}}(x)$ appropriately whenever a dual constraint for some $(x,y) \in \neigh_k$ is found to be violated. Then we call the algorithm with the modified cost $\hat{c}(x,y) = c(x,y) - \alpha_{k+1,\tn{init}}(x) - \beta_{k+1,\tn{init}}(y)$. $\hat{c}$ has the same primal optimizers as $c$ and the dual optimizers are related by adding / subtracting the initial dual variables.
			
			We refer to the variants with and without adaptive re-initialization by \texttt{init} and \texttt{noinit}.
		\end{remark}

		\begin{remark}[Truncation]
			\label{rem:Truncation}
			Note that the algorithms \texttt{LEMON-NS} and \texttt{LEMON-CS} internally only operate on integer data. So for application of these algorithms all problem data had to be truncated to a discrete grid of values. We set the resolution of this grid to $10^{-9}$ (i.e.\ a measure of unit mass was approximated by $10^9$ mass atoms) which implies a high practical resolution.
		\end{remark}

	\paragraph{Test Data.}
	Most experiments were performed on regular grids in $\R^2$ with full support to ensure existence of the full 4-neighbourhood which was chosen for the shielding candidates $\shield(\cdot)$, see also Remark \ref{rem:ShieldingCartesianGrids}.
	Assuming full support is also common for continuous solvers (e.g.~\cite{OptimalTransportWarping,ObermanMongeAmpere2014}). This can be ensured by adding a small constant measure. We observed that the solvers could handle very small constants and that thus the distortion of the problems was negligible.
	We considered grid sizes between 50 $\times$ 50 and 512 $\times$ 512, so the cardinalities of $X$ and $Y$ ranged between $2.5 \cdot 10^3$ and $2.6 \cdot 10^5$ and the dimensions of the full coupling spaces between $6.3 \cdot 10^6$ and $6.9 \cdot 10^{10}$. We considered transports between grids of equal size, i.e.\ $|X|=|Y|$.
	The tested measures were generated by adding random Gaussians with eigenvalues of the covariance matrix ranging between $1.8$ and $\approx 100$. Some densities were then multiplied with discontinuous masks. So the test problems contain both smooth parts as well as strong Dirac-like local concentrations of mass and sharp discontinuities. We consider these to be challenging problems (see Fig.~\ref{fig:ExpJacobian}), representing a wide range of practical applications.
	All reported numbers are averaged over a collection of test problems (10 to 50 problems, depending on size).
	
	\subsection{Comparing Different Internal Solvers}
		\label{sec:NumericsRuntimes}
		\begin{figure}[hbt]%
			\centering%
			\includegraphics{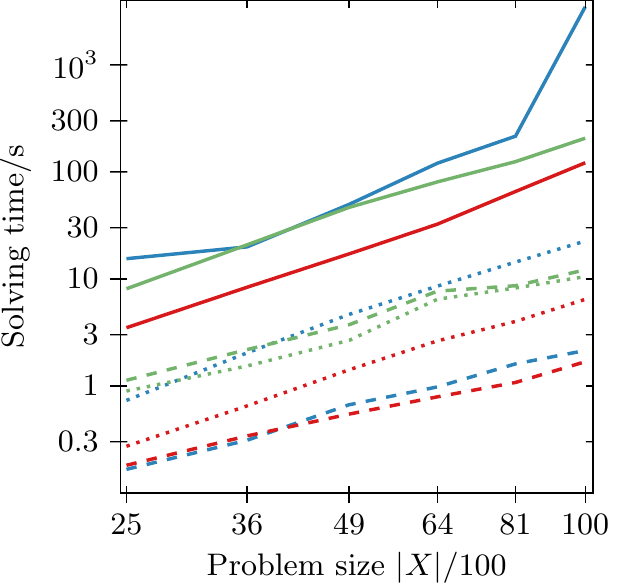}%
			\hfill %
			\includegraphics{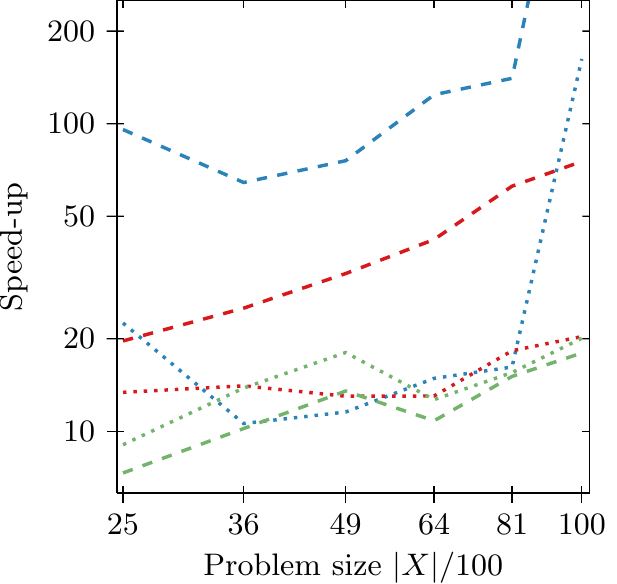}%
			\hfill %
			\includegraphics{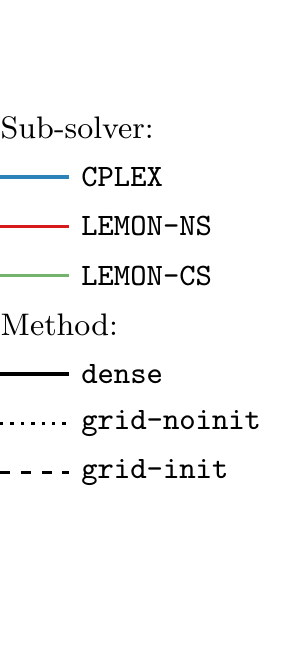}%
			\caption{%
				\textbf{Left:} Overall run-times of \texttt{dense} and sparse methods (\texttt{grid-init} and \texttt{grid-noinit}) for different internal sub-solvers. \textbf{Right:} Speed-up ratio of sparse solvers w.r.t.\ dense solver. \\
				Clearly the sparse multi-scale methods outperform the naive dense solvers. For \texttt{LEMON-CS} \texttt{init} does not accelerate solving relative to \texttt{noinit}. For \texttt{LEMON-NS} and in particular for \texttt{CPLEX} \texttt{init} leads to significantly further reduced run-times, resulting in speed-ups of around two orders of magnitude and more.}
			\label{fig:ExpRuntimesBasic}
		\end{figure}
		We compare the run-times of the naive dense algorithms (\texttt{dense}) with using them as internal solvers $\funcSolve(\cdot)$ in Algorithm \ref{alg:ShortCutSolver} combined with \ref{alg:MultiScale}. For simplicity, for $\funcShield(\cdot)$ we choose the grid-based method (\texttt{grid}). For the multi-scale timing we sum the time it takes to solve all levels, from coarse to fine. The observed run-times and speed-ups (with (\texttt{init}) and without (\texttt{noinit}) adaptive re-initialization, Remark \ref{rem:ReInitialization}) are illustrated in Fig.~\ref{fig:ExpRuntimesBasic}. All reported run-times were obtained on a single core of an Intel Xeon E5-2697 processor at 2.7 GHz.
		
		The solving times per iteration in Algorithm \ref{alg:ShortCutSolver} are plotted in Fig.~\ref{fig:ExpBasicStepsolving}. For \texttt{noinit} the observed acceleration roughly ranges between 10 and 20 (as reported in \cite{Schmitzer-SSVM2015}). For \texttt{init} the trick described in Remark \ref{rem:ReInitialization} for the \texttt{LEMON-*} algorithms does not seem to have any significant effect on the run-time of \texttt{LEMON-CS} whereas with its help \texttt{LEMON-NS} can be accelerated by up to almost two orders of magnitude. For \texttt{CPLEX} the dedicated re-initialization method allows a speed-up of well over $10^2$. As the dense \texttt{CPLEX} solver is particularly slow for problem sizes of $10^4$, the average time ratio even exceeds $10^3$ for this problem size (\texttt{CPLEX} internally operates on \texttt{double} data, so a potential explanation for this divergence may be given by Remark \ref{rem:Truncation}). In particular for the re-initialized variants of \texttt{CPLEX} and \texttt{LEMON-NS} the reported speed-up consistently increases with problem size.
		
		For simplicity and due to the availability of a dedicated re-initialization method, we focus on \texttt{CPLEX} in the following experiments, keeping in mind that similar solving times may be obtained with \texttt{LEMON-NS} while \texttt{LEMON-CS} performs less well.
		
		\begin{figure}[hbt]
			\centering%
			\subfloat[]{%
			\label{fig:ExpBasicStepsolving}%
			\includegraphics{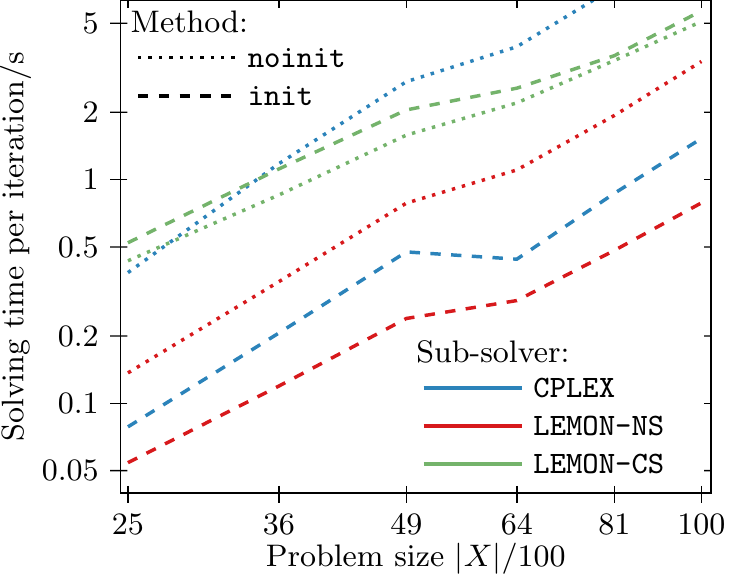}%
			}%
			\hfill %
			\subfloat[]{%
			\label{fig:ExpBasicSparsity}%
			\includegraphics{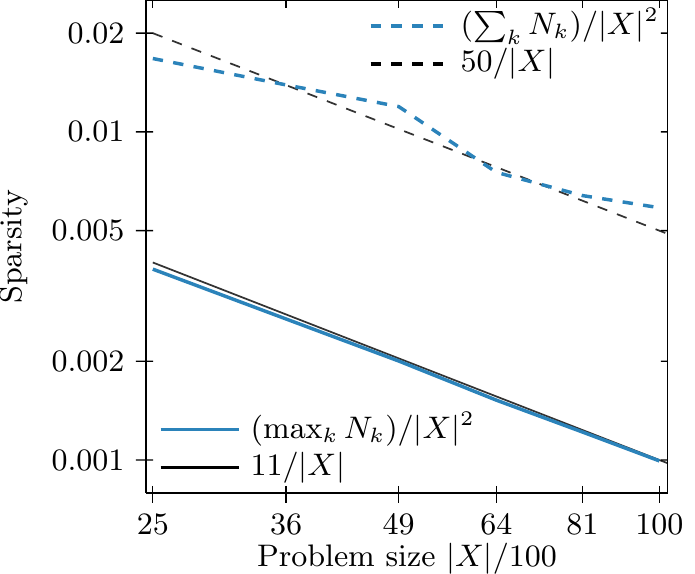}%
			}%
			\caption{%
				\textbf{Left:} %
				Solving time for a single iteration in Algorithm {\protect \ref{alg:ShortCutSolver}} at the finest scale level. As discussed in Fig.~{\protect \ref{fig:ExpRuntimesBasic}}, \texttt{init} is efficient for \texttt{LEMON-NS} and \texttt{CPLEX}. %
				\textbf{Right:} Sparsity of shielding neighbourhoods for \texttt{CPLEX} (plots for \texttt{LEMON-*} look very similar). Maximum ($\max_k |N_k| / |X \times Y|$) and cumulative ($\sum_k |N_k| / |X \times Y|$) number of variables, divided by total number of variables (recall we chose $|X|=|Y|$) in the shielding neighbourhoods of the iterations in Algorithm {\protect \ref{alg:ShortCutSolver}} at finest scale. %
				The fraction of relevant variables reduces with increasing problem size. %
				The black lines give $\mc{O}(1/|X|)$ for comparison, indicating that the number of variables in $N_k$ \emph{per $x \in X$} is roughly constant. %
				This illustrates that the shielding mechanism works as sketched in {\protect \thref{thm:ShieldingYHatBound}}. %
				}
		\end{figure}
	\subsection{Sparsity and Number of Iterations}
		\label{sec:NumericsSparsity}
		\begin{figure}[hbt]
			\centering
			\includegraphics{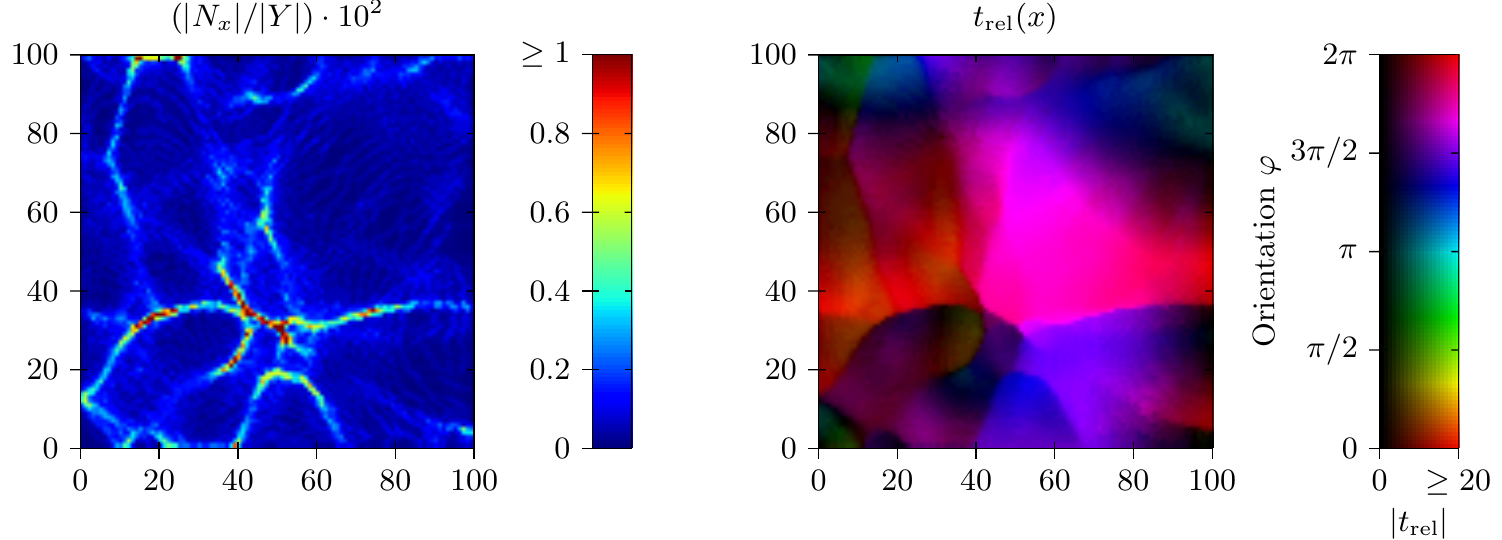}%
			\caption{Analysis of shielding neighbourhood for a test problem with $|X|=|Y|=100^2$. \textbf{Left:} For $x \in X$ let $N_x = N_k \cap (\{x\} \times Y)$, the $x$-row of the shielding neighbourhood $N_k$. Heat-map of $|N_x|/|Y|$ over $X$ after one iteration of Algorithm {\protect \ref{alg:ShortCutSolver}} ($k=2$). Color range from $0$ to $1\%$. %
			\textbf{Right:} Average relative displacement map $t_{\tn{rel}}(x)=\sum_y (y-x)\,\pi_k(x,y)/\mu(x)$ in polar coordinates over $X$. \\
			The map $t_{\tn{rel}}$ contains long displacements ($>20$ pixels), sharp discontinuities and strong compression / expansion (the Jacobian determinant of $t_{\tn{rel}}$ ranges from $10^{-2}$ to $10^2$), indicating a challenging test problem.
			We see that the sparsity adapts locally to the spatial regularity of the assignment. In regions where $t_{\tn{rel}}$ is regular only few neighbours per element are necessary. %
			$|N_x|$ rises at the discontinuities of $t_{\tn{rel}}$, but this effect is local.
			Note that there is no rise in $|N_x|$ at the boundaries of the image (cf.~Remark {\protect \ref{rem:ShieldingCartesianGrids}}).}%
			\label{fig:ExpJacobian}%
		\end{figure}
		The demonstrated speed-up relies on the sparsity of $\neigh_k$ which also reduces the memory requirement of the algorithm. Our numerical findings are presented in Fig.~\ref{fig:ExpBasicSparsity}. The ratio $|\neigh_k|/|X \times Y|$ is consistently decreasing with increasing problem size and we observe $|\neigh_k| = \mc{O}(|X|)$ as discussed in Sect.~\ref{sec:Shielding} (see in particular \thref{thm:ShieldingYHatBound}). The assumption that $\pi_k$ is spatially regular therefore seems appropriate.
		Even the sum $\sum_k |\neigh_k|$ over all iterations of Algorithm \ref{alg:ShortCutSolver} does scale as $\mc{O}(|X|)$. On the test data the median number of iterations was $5$, the $95\%$ quantile was $8$ iterations, thus numerically confirming the efficiency of the multi-scale scheme for obtaining good initializations (cf.~Sect.~\ref{sec:AlgorithmHierarchy}).
		
		Figure \ref{fig:ExpJacobian} gives an impression of the structure of $\neigh_k$ during execution of Algorithm \ref{alg:ShortCutSolver}. We see that $\neigh_k$ locally adapts to the regularity of the assignment: in regular areas only few elements in $\neigh_k$ are needed per element $x \in X$. In irregular regions the size of the neighbourhood increases, but this is only a local effect. The regular regions are not affected by this. This indicates that the global regularity assumption on $\pi$ in \thref{thm:ShieldingYHatBound} could be relaxed to a local variant in a potential complexity analysis of the method.
		
	\subsection{Larger Problems}
		\begin{figure}[hbt]%
		\centering%
		\includegraphics{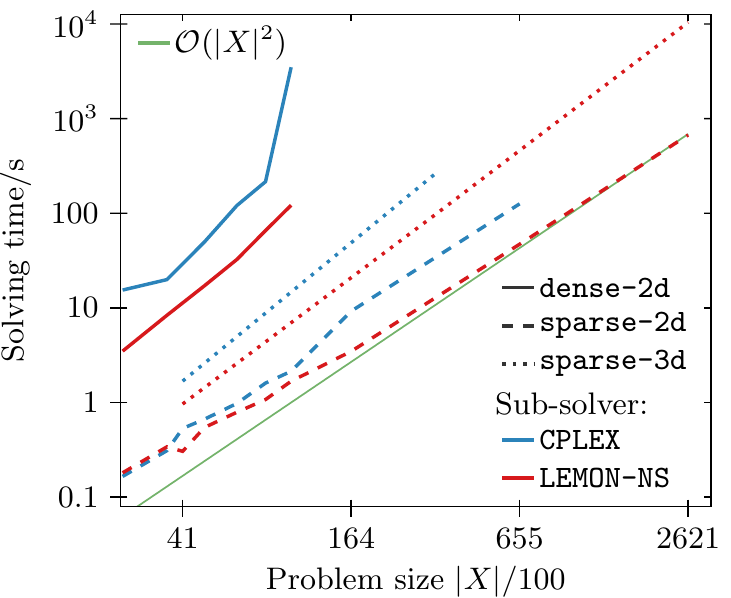}%
		\hfill %
		\includegraphics{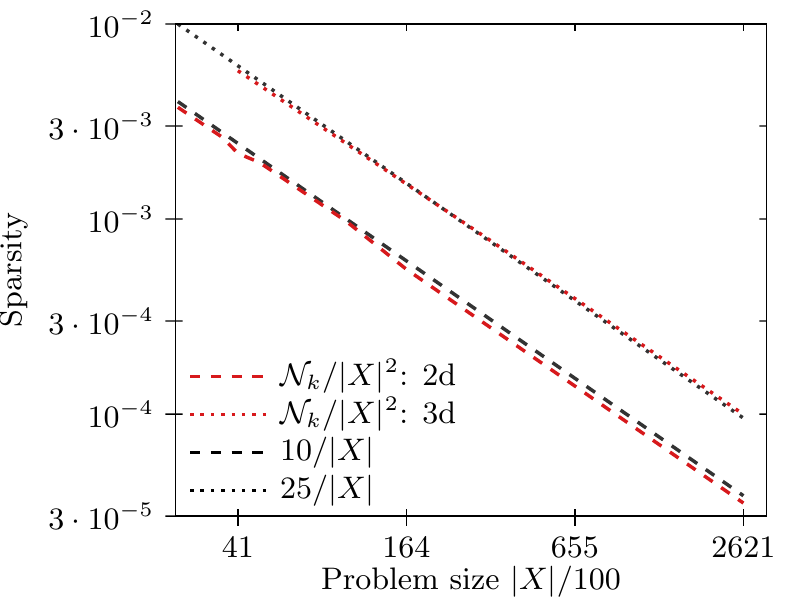}%
		\caption{%
			\textbf{Left:} Run-times of the sparse solver (\texttt{grid-init}) for 2d-data up to image size  $512 \times 512$ ($|X|=|Y| \approx 2.6 \cdot 10^5$) and 3d-data up to cube size $64 \times 64 \times 64$. The dense 2d run-times (from Fig.\ \ref{fig:ExpRuntimesBasic}) are given for comparison (the dense 3d run-times can be assumed to be similar). %
			The green line marks run-time scaling of order $\mc{O}(|X|^2)$, indicating that the performance of the sparse solver is super-linear (approximately quadratic) in the marginal size. Therefore, also the sparse solver eventually becomes impractically slow. %
			\textbf{Left:} Sparsity of the shielding neighbourhoods for \texttt{LemonNS}, analogous to Fig.~\ref{fig:ExpBasicSparsity}. On 3d-data a similar behaviour is observed, with a higher number of variables per $x \in X$, owing to the higher dimension of the base space.}
		\label{fig:ExpLarge}
	\end{figure}

	The maximal problem size in the experiments presented in Sect.~\ref{sec:NumericsRuntimes} is determined by the limitations of the dense solvers: above $|X|=|Y|=100 \times 100$ memory demand soon exceeds the capabilities of standard workstations and run-times become too large even for testing purposes.
	In this section we study the limits of the sparse solvers.
	We test two-dimensional data on grids up to size 512 $\times$ 512 and three-dimensional cubes up to size $64 \times 64 \times 64$, using the squared Euclidean distance as cost. The results are summarized in Fig.\ \ref{fig:ExpLarge}.
	
	The observed sparsity for two dimensions is consistent with Fig.~\ref{fig:ExpBasicSparsity}. As expected from Sect.~\ref{sec:ShieldingSqrEuclidean}, in three dimensions the scaling is similar but with more coupling variables per $x \in X$.
	With the presented algorithm 2d-problems up to about 256 $\times$ 256 can be solved conveniently. For three dimensions the performance is weaker, due to the increased number of sparse variables.

	While memory is no longer a practical limitation, the run-time scales super-linearly with the marginal size $|X|$, thus eventually also becoming impractical.	
	A direction of further research might be the consistent splitting of one large problem into smaller parts, avoiding the super-linear scaling and using multiple CPUs in parallel.

	\subsection{Comparing Different Shielding Construction Methods}
		\begin{figure}[hbt]%
			\centering%
			\includegraphics{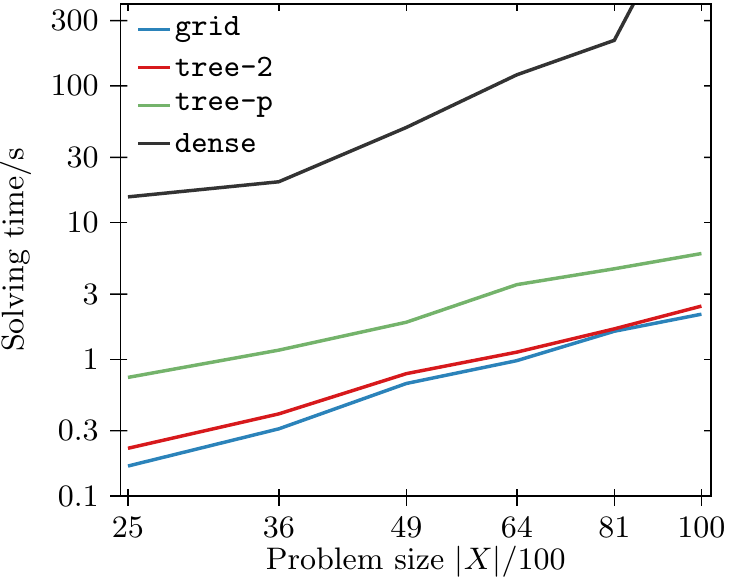}%
			\hfill %
			\includegraphics{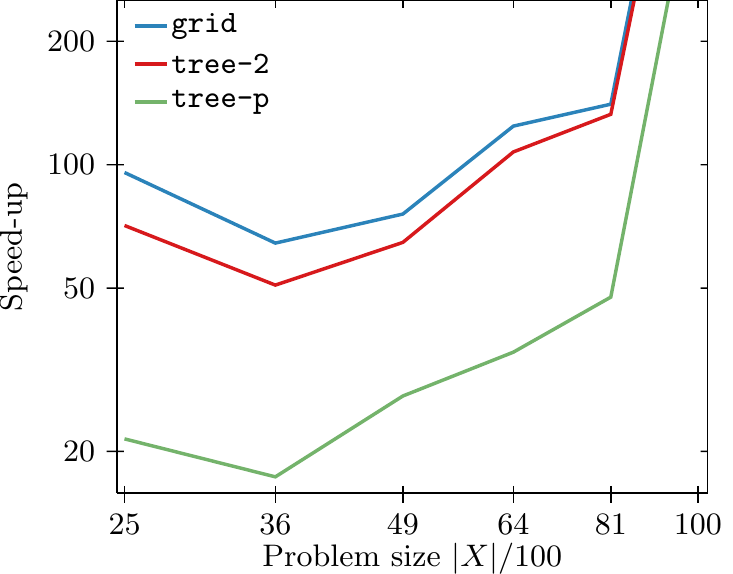}%
			\caption{\textbf{Left:} Comparing run-time of dense approach and sparse algorithm for different shielding methods. %
			\textbf{Right:} Implied speed-up relative to the dense solver. \\%
			As expected, \texttt{grid} is the most efficient variant, however \texttt{tree-2} takes only slightly more time. The less precise bound used in \texttt{tree-p} is somewhat slower but still achieves a speed-up of about 50 for $|X|=|Y|=8100$, consistently increasing with problem size.}%
			\label{fig:ExpShieldingTimes}%
		\end{figure}
		For the squared Euclidean distance on regular grids we have presented three methods to construct shielding neighbourhoods: the direct method via exploiting the grid structure (\texttt{grid}, Remark \ref{rem:ShieldingCartesianGrids}), the hierarchical bound \texttt{tree-2} (\thref{thm:ShieldingPsiEuclidean}) and as a special case of the hierarchical bound \texttt{tree-p} (\thref{thm:ShieldingPsiPEucl}). We will now compare the efficiency of these methods. Note that all three will eventually produce the same neighbourhoods. So the solving part of Algorithm \ref{alg:ShortCutSolver} is not affected.
		The total run-times and resulting speed-ups are displayed in Fig.~\ref{fig:ExpShieldingTimes}. As expected, \texttt{grid} performs best, \texttt{tree-2} is only a little slower. The bound \texttt{tree-p}, containing more approximations, is less tight and therefore less efficient but still provides a significant speed-up over the dense method.
		
		Additional insight can be gained from Fig.~\ref{fig:ExpShieldingTShielding} which displays the average time to construct a shielding neighbourhood and Fig.~\ref{fig:ExpShieldingNShielding} which displays the number of evaluations of the hierarchical bound $\hat{\psi}_{(\cdot,\cdot)}$, \eqref{eq:ShieldingConditionPsiHat}, during the hierarchical search of $\hat{Y}$ by Algorithm \ref{alg:FindYHat}, over the whole construction of a shielding neighbourhood (i.e.\ not just for one $x \in X$).
		It roughly scales as $\mc{O}(|X|)$, i.e.\ the hierarchical search is in fact more efficient than a naive dense iteration over $Y$.
		\begin{figure}[hbt]
			\centering%
			\subfloat[]{%
			\includegraphics{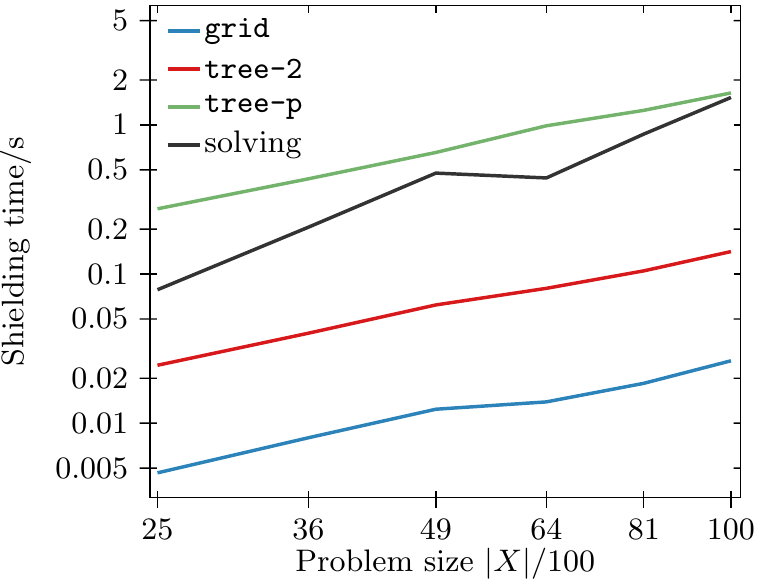}%
			\label{fig:ExpShieldingTShielding}%
			}%
			\hfill
			\subfloat[]{%
			\includegraphics{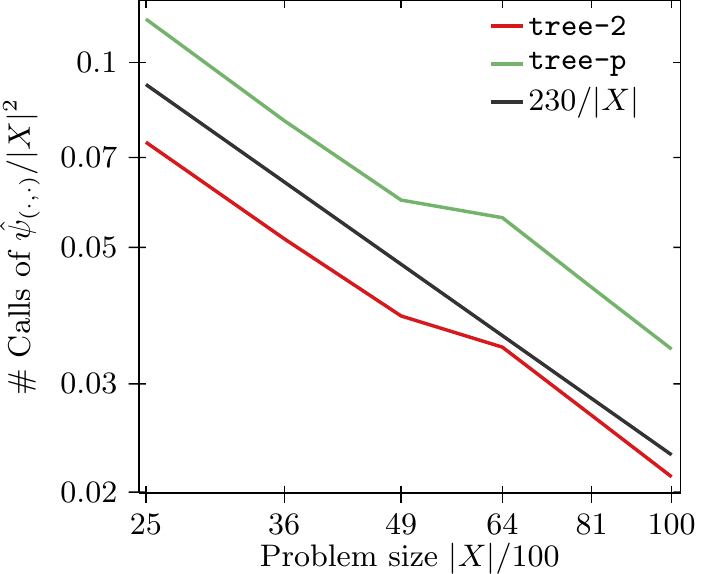}%
			\label{fig:ExpShieldingNShielding}%
			}%
			\caption{%
				\textbf{Left:} Time for construction of shielding neighbourhood for \texttt{grid}, \texttt{tree-2} and \texttt{tree-p}. Time for subsequent sparse solving is given for comparison. %
			\textbf{Right:} Number of evaluations of the bound $\hat{\psi}_{(\cdot,\cdot)}$ during construction of a shielding neighbourhood, divided by $|X \times Y|$. The relative number of these calls decreases with problem size. The black line represents a scaling like $\mc{O}(|X|)$. This means that the number of calls per element $x \in X$ is roughly constant and thus the hierarchical search is more efficient than the naive full iteration over all pairs. The kink between $4900$ and $6400$ is due to a new layer being added in the hierarchical partition.}
			\label{fig:ExpShielding}
		\end{figure}

	\subsection[\textbar x-y\textbar \textasciicircum p for various p]{$|x-y|^p$ for various $p$.}
		\label{sec:NumericsPEucl}
		Let us now compare \texttt{tree-p} for different exponents $p$. In Fig.~\ref{fig:ExpPEucl} results for various values of $p$ in $[1.01,2.5]$ are displayed. As expected, the dense solving time does not depend significantly on the exponent.
		
		For $p \searrow 1$ the sparse solver becomes increasingly slower, in the extreme case $p=1.01$ even exceeding the run-time of the dense solver. This can be understood by looking at Fig.~\ref{fig:ShieldingPEuclIllustration}. As $p \searrow 1$, the area covered by one shielding candidate decreases, leading to an increase in the calls to the hierarchical bound $\hat{\psi}_{(\cdot,\cdot)}$ as well as larger shielding neighbourhoods $\mc{N}_k$.
		As $p$ moves away from $1$, the number of calls and variables quickly becomes almost constant, leading to similar run-times.
		
		An exception is the case $p=2$. A detailed study of the underlying code revealed that the difference in run-time stems largely from the \texttt{c++} function \texttt{pow()}, which is only required for $p \neq 2$ (cf.~\eqref{eq:ShieldingConvexPEuclSummary}), thus the case $p=2$ takes somewhat less time, which accumulates over the number of all queries.
		This sensitivity to small performance differences in the code highlights the potential of further research on bounds like \thref{thm:ShieldingPsiPEucl} to obtain simple and yet sufficiently accurate expressions, as well as computationally streamlined implementation thereof.
		\begin{figure}[hbt]
			\centering%
			\subfloat[]{%
			\includegraphics{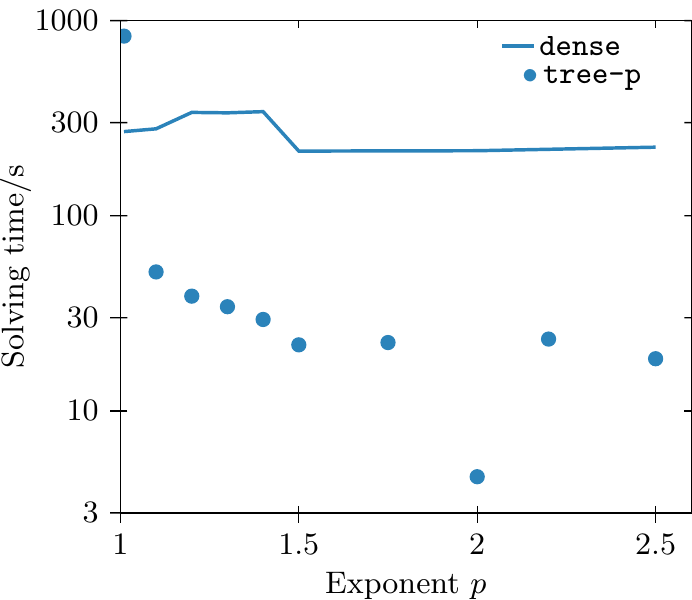}%
			}%
			\hfill %
			\subfloat[]{%
			\includegraphics{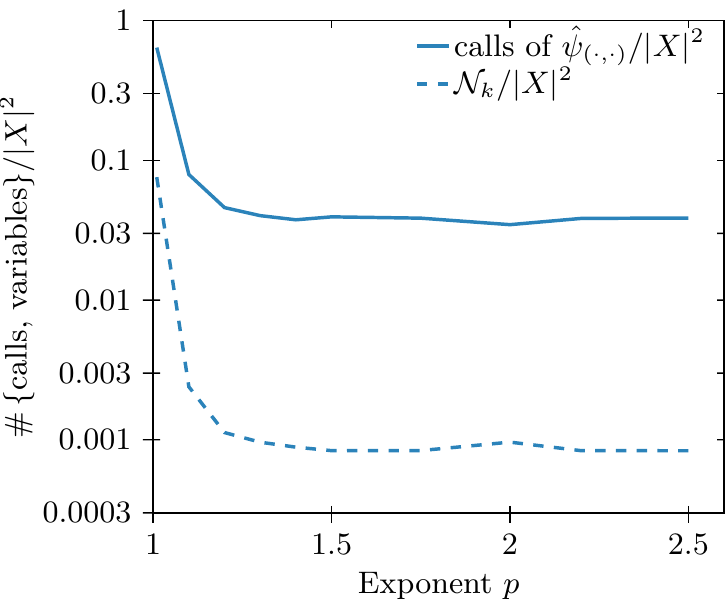}%
			\label{fig:ExpPEuclQueries}
			}%
			\caption{\textbf{Left:} Solving times for cost $c(x,y) = |x-y|^p$ for different $p$ for \texttt{dense}, and \texttt{tree-p} for $p \in [1.01,2.5]$. %
			\textbf{Right:} Number of evaluations of the bound $\hat{\psi}_{(\cdot,\cdot)}$ during construction of a shielding neighbourhood and cardinality of $\neigh$, divided by $|X| \times |Y| = |X|^2$. See also Fig.~{\protect \ref{fig:ExpShielding}}.\\ %
			For $p \searrow 1$ the sparse method becomes slower, caused by an increasing number of bound evaluations and neighbourhood variables. For $p \gg 1$ the situation stabilizes, resulting in almost constant number of calls, variables and run-time.
			However, for $p=2$ the computation time of such a query is slightly lower (see text), thus resulting in an overall lower run-time.}
			\label{fig:ExpPEucl}%
		\end{figure}

	\subsection{Noisy Costs}
		An important aspect of the presented discrete framework is the robustness towards noisy costs (Sect.~\ref{sec:ShieldingNoise}). Here we consider the following set-up: $\cGeo = |\cdot-\cdot|^2$, $\cNoise$ is randomly sampled from $[0,1]$ (in the hierarchical cost functions $c_k$, $k>0$ this contribution can be ignored). For the Lipschitz part we chose:
		\begin{align*}
			\cLip(x,y) & = \frac{k_{\tn{mag}}}{2\pi} \cdot \sin\left(\frac{2\pi}{k_{\tn{mag}}} \cdot \la k(y), x\ra \right) \\
			\intertext{with}
			k(y) & = (\cos \varphi(y), \sin \varphi(y))^\T, \qquad \varphi(y) = \la (1,1)^\T, y \ra\,,
			\qquad k_{\tn{mag}}=20\,.
		\end{align*}
		So $\cLip(x,y)$ is a sine in $x$ for each fixed $y$, its orientation given by $\varphi$ which is simply chosen to provide some `random' angles for each $y \in \R^2$.
		
		The corresponding weights $\eta$ and $\lambda$ where set to all combinations in $(\eta,\lambda) \in \{0,5,10,15\}^2$. For $|x_A-x_s|=1$, as on the finest level of our Cartesian grid, this corresponds to additional slacks ranging from $0$ to $30$. This means that, compared to the clean squared Euclidean distance, a particular assignment could be distorted by as much as 30 pixels, which is a considerable fraction of the underlying image size.
		\begin{figure}[hbt]
			\centering %
			\includegraphics{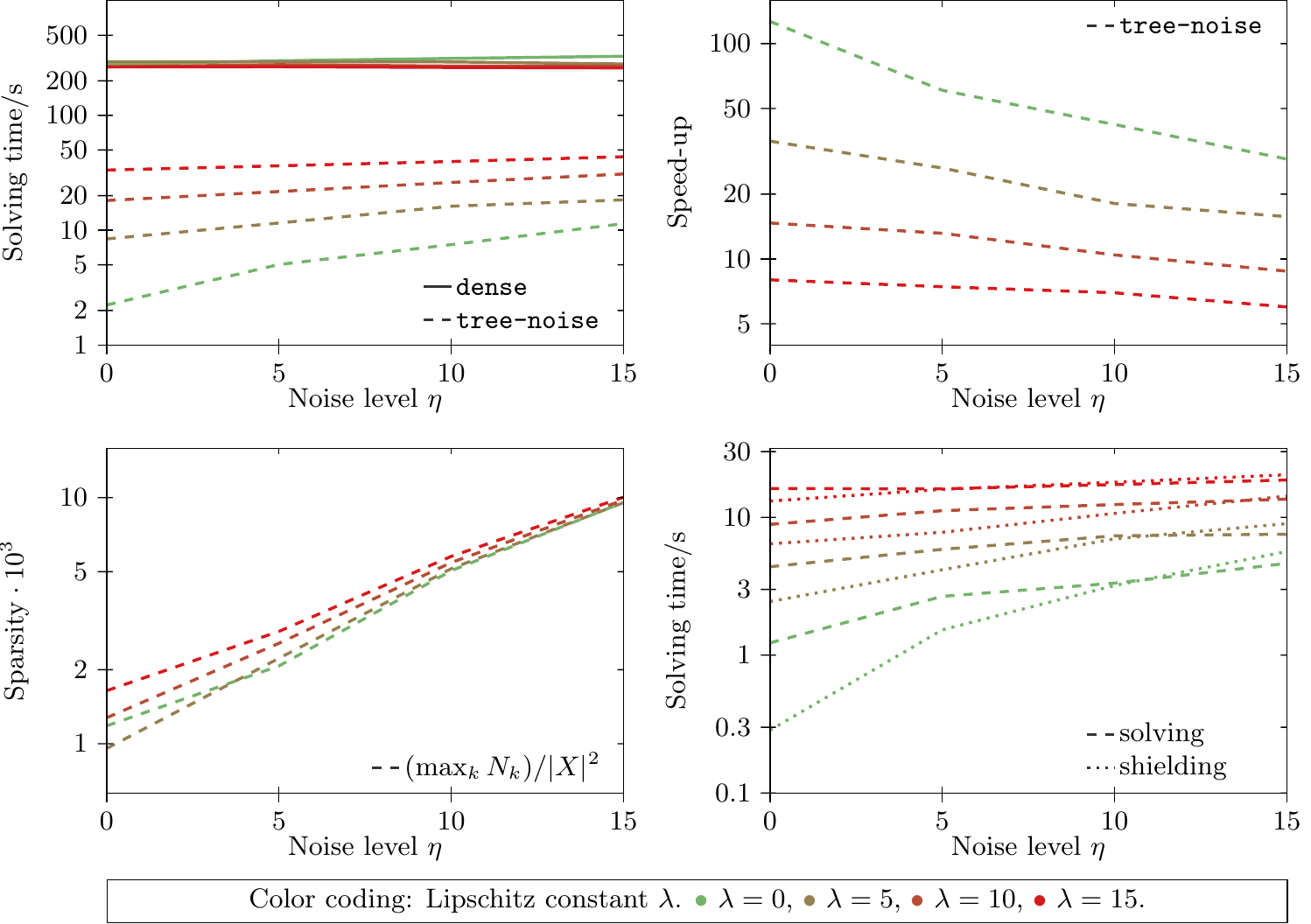} %
			\caption{Noisy costs on $|X|=|Y|=90^2$. \textbf{Top left:} Total solving time on noisy costs for dense and sparse algorithms. As expected, the sparse solving time increases with both types of noise (random and Lipschitz), while the dense solving time is hardly affected.
			\textbf{Top right:} Implied speed-up relative to dense solver. It decreases with increasing noise levels. But the method remains applicable and effective: even for $\eta + \lambda \cdot |x_A-x_s|=30$ (which implies that local assignments can be distorted by as much as 30 pixels relative to the squared Euclidean distance) the speed-up is almost one order of magnitude. %
			\textbf{Bottom left:} Sparsity of shielding neighbourhoods $\neigh_k$. The Lipschitz distortion has less impact on the variable count than the random noise. (Not shown: the scaling of $|\neigh_k|$ for different problem sizes is still $\mc{O}(|X|)$.) %
			\textbf{Bottom right:} Time required for solving and constructing shielding neighbourhoods on finest scale level (sum over all iterations). With increasing noise levels the additional slack in the bound $\hat{\psi}_{(\cdot,\cdot)}$ (cf.\ Sect.\ {\protect \ref{sec:ShieldingNoise}}) requires more queries and thus increases the shielding time.}
			\label{fig:ExpNoise}
		\end{figure}
		The corresponding numerical findings for $|X|=|Y|=90^2$ are presented in Fig.~\ref{fig:ExpNoise}. The overall solving-time increases with $\eta$ and $\lambda$, but even for the noisiest case that we considered, there remained a speed-up of almost one order of magnitude. The corresponding neighbourhood sizes increase with $\eta$. The value of $\lambda$ seems to have little impact. The neighbourhood sizes still scale as $\mc{O}(|X|)$.
		
		It is no surprise that the sparse method becomes less efficient with increasing noise. However, this decrease comes gradually. Unlike continuous solvers, there is no immediate breakdown when noise is added. The method remains applicable and useful.
		
	\subsection{Sphere}
		\begin{figure}[hbt]
			\centering
			\includegraphics{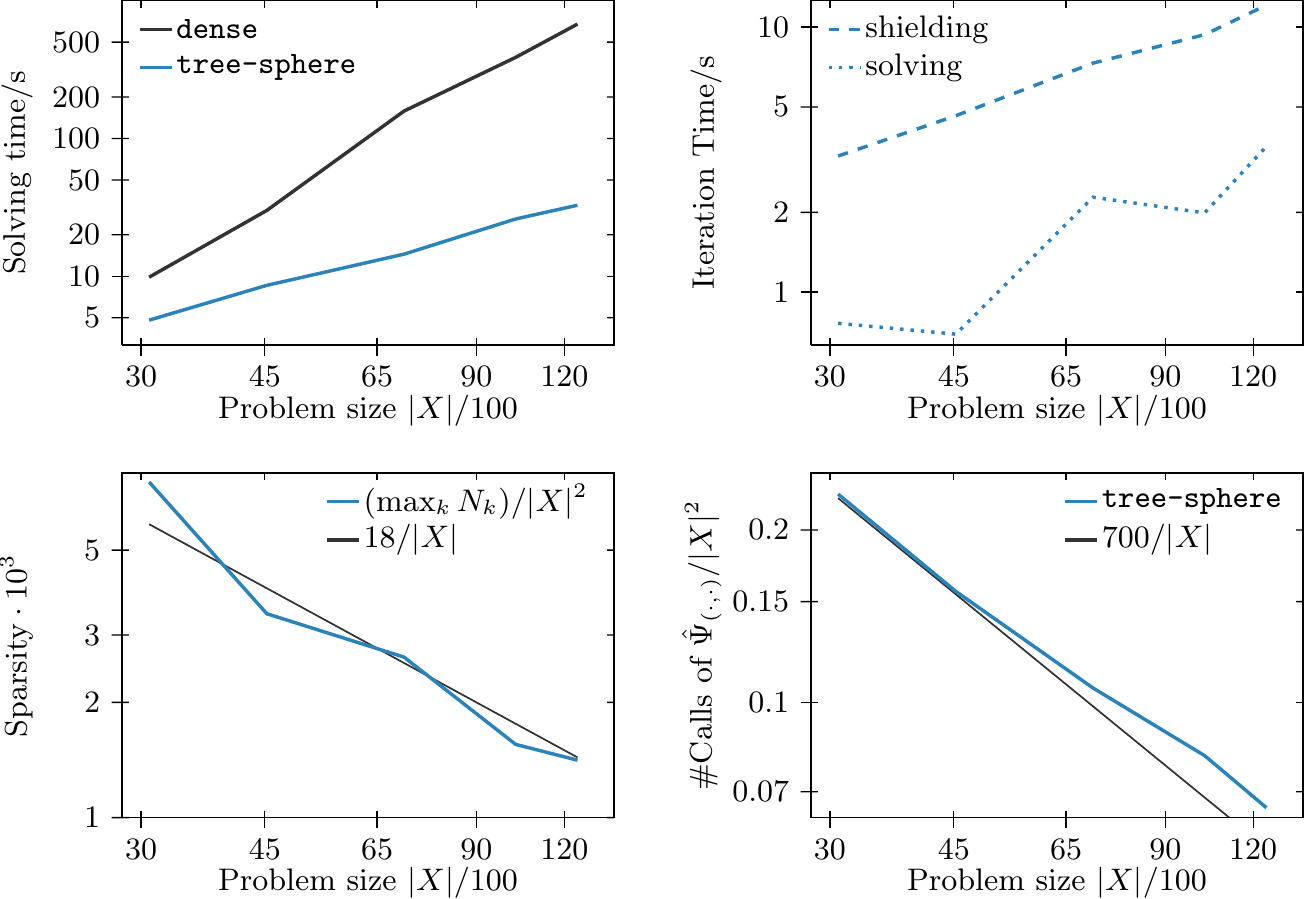}
			\caption{Sparse solving on the sphere. \textbf{Top left:} Average absolute solving time of dense and sparse solver. While the speed-up is not as strong as on $\R^n$, it still increases consistently with problem size and is well above one order of magnitude for the largest problem considered. %
			\textbf{Top right:} Time for construction of shielding neighbourhood and subsequent sparse solving. %
			\textbf{Bottom left:} Sparsity of shielding neighbourhood. Black line gives scaling of $\mc{O}(|X|)$. %
			\textbf{Bottom right:} Number of evaluations of the bound $\hat{\psi}_{(\cdot,\cdot)}$ during construction of a shielding neighbourhood, divided by $|X \times Y|$. See also Fig.~{\protect \ref{fig:ExpShielding}}. \\%
			In this setting the shielding time exceeds the solving time (see text for explanation). However, the low solving time and the plots on the bottom row show that the basic concept of shielding neighbourhoods is useful in this more general setting.}
			\label{fig:ExpSphere}
		\end{figure}
		We sample points from the sphere $\sphere$ and create test densities similar to $\R^n$ but instead of Gaussians we add radial `pseudo-Gaussians' of the form $f(x) = A \cdot \exp(-d(x_0,x)^2/(2\sigma))$.
		
		Numerical results using the hierarchical bound \texttt{tree-sphere}, \thref{thm:ShieldingPsiSphere}, are summarized in Fig.~\ref{fig:ExpSphere}. Again, we observe a speed-up, increasing with problem size, well exceeding one order of magnitude for the largest test problems.
		In this setting the time to construct shielding neighbourhoods exceeds the time required to solve the sub-problems. This is mostly due to the use of several trigonometric functions in \thref{thm:ShieldingPsiSphere}. As we have learned in the $|x-y|^p$-experiments, even little computational effort per call adds up over the whole execution of the algorithm.
		Note however, that the sparsity ratios are comparable to that in $\R^n$ (Fig.~\ref{fig:ExpBasicSparsity}) and that also the relative number of calls of $\hat{\psi}_{(\cdot,\cdot)}$ is decreasing, indicating the usefulness of the concept of shielding neighbourhoods even in this more general setting.
		
		From this follows that further significant reduction of the run-time may be obtained by computationally streamlining the hierarchical bounds, without needing to change the solver.


\section{Discussion and Conclusion}
\label{sec:Conclusion}

\paragraph{Symmetry.}
The concept of short-cuts and the shielding condition are symmetric in $X$ and $Y$. Symmetry is broken in \thref{thm:ShieldingImpliesShortcut}. A more general, symmetric version of \thref{thm:ShieldingImpliesShortcut} can be established, providing a larger class of short-cuts.
Consequently one could weaken the assumptions on shielding neighbourhoods in Def.~\ref{def:ShieldingNeighbourhood}.
The advantage of the non-symmetric version chosen in this article becomes apparent in Algorithm \ref{alg:ConstructShielding} where the main loop is only over $X$. To exploit the increased flexibility of the symmetric formulation one would in general have to iterate over $X \times Y$, thus increasing the complexity of the construction.
Moreover, as we have discussed for example in \thref{thm:ShieldingYHatBound} and as demonstrated numerically, we already have $|\neigh_k| = \mc{O}(|X|)$, which will not be improved upon by a symmetric variant.

\paragraph{Comparison with \cite{ObermanOptimalTransportationLP2015}.}
	The approach presented in \cite{ObermanOptimalTransportationLP2015} is similar to \funcSolveShortCuts\ combined with \funcSolveMultiScale\ (Algorithms \ref{alg:ShortCutSolver} and \ref{alg:MultiScale}), in particular to the preliminary version presented in \cite{Schmitzer-SSVM2015}. Discrete optimal transport problems are solved hierarchically from coarse to fine, while only looking at sparse sub-problems. 
	The sparse region of interest on the finer scale is generated by taking the support of the coarse solution and adding neighbouring elements (in the product space). The intuition behind this is similar to the discussion presented in Section \ref{sec:ShortCutsIntuition}.
	This procedure is justified informally by arguing that the discrete optimal solution does converge weakly to the underlying optimal continuous coupling, which is known to live on a graph and by a sensitivity analysis of linear programming. However, no rigorous proof for the global optimality of the obtained solutions is given.
	The main focus of \cite{ObermanOptimalTransportationLP2015} lies in providing a fast practical method, which is demonstrated on a large set of numerical examples, involving the approximate reconstruction of the continuous Monge map, computation of free boundaries in optimal partial transportation and Wasserstein barycenters.
	The reported run-times are comparable to our results.
	
	Adding the neighbours of the support in the product space can be seen as a symmetrized version of (\ref{item:ShieldAlgYNeigh}) in Algorithm \ref{alg:ConstructShielding} (cf.\ the above discussion on symmetry). However, there is no equivalent to (\ref{item:ShieldAlgYLoop}).
	For the squared Euclidean distance over a Cartesian grid and for smooth marginal densities we found that this works well in practice: after several iterations with the intuitive approach the solver would typically converge to the globally optimal solution, very similar to the full shielding method. This can be understood by looking at the geometric interpretation of the shielding condition in that case (Fig.\ \ref{fig:ShieldingEuclidean}): when the marginals are smooth, the assignment is usually spatially regular and therefore the interior of the polytope, which is not shielded, is very small (and thus step (\ref{item:ShieldAlgYLoop}) would not add many elements).
	However, for inhomogeneous densities, for $p \neq 2$ and for point clouds that do not lie on regular Cartesian grids, we observed that the intuitive method was prone to getting stuck in slightly sub-optimal couplings.
	While this is probably sufficient for many practical applications, it should be noted that at least for the standard scenario of squared Euclidean distance over Cartesian grids, the implementation of (\ref{item:ShieldAlgYLoop}) is very simple (see Remark \ref{rem:ShieldingCartesianGrids}) and thus global optimality can be ensured with little additional effort.
	
\paragraph{Complexity.}
Determining the overall complexity of Algorithm \ref{alg:MultiScale} is an important open question. It depends on various factors, including complexity characteristics of the internal solver, the sizes of the neighbourhoods $\{\neigh_k\}_k$ and the number of outer iterations at each level in Algorithm \ref{alg:ShortCutSolver}.
As pointed out in Remark \ref{rem:ShieldingYHatBoundInterpretation}, using statements like \thref{thm:ShieldingYHatBound} to control the maximum size of the sparse neighbourhoods would require sufficiently strong a priori bounds on the regularity of the optimal couplings $\{\pi_k\}_k$ in Algorithm \ref{alg:ShortCutSolver} over \emph{multiple iterations}.
To our knowledge such estimates are currently not available and thus a complete complexity analysis cannot be given here.
We have however provided numerical evidence that the neighbourhood sizes scale as predicted by \thref{thm:ShieldingYHatBound} and that the overall run-time is reduced significantly.
Note further that this uncertainty only affects the efficiency of the scheme. Global optimality of the final result is rigorously established and is guaranteed even in cases where the method is very slow.

\paragraph{Conclusion.}
	Dense optimal transport problems appear in many applications. Often it is known that the optimal coupling will have sparse support, but existing discrete solvers are not able to take advantage of this sparsity and thus scale poorly on large problems.
	Our paper provides a way of verifying global optimality of a coupling by only looking at a suitable shielding subset of the full product space. This can be interpreted as discrete equivalent for well-known continuum results.
	Based thereon we proposed an algorithm that provably solves dense problems globally via a sequence of sparse problems. This algorithm can be combined with coarse-to-fine multi-scale methods.
	A part of this algorithm is the efficient construction of sparse shielding sets, which must be adapted to the cost function. We explicitly discuss this construction for several costs on $\R^n$ (including the squared Euclidean distance), the sphere and noisy variants thereof and gave some indications on why one can expect these sets to be sparse.
	The efficiency of the scheme in terms of run-time and sparsity was demonstrated numerically.
	Our algorithm thus allows the application of discrete solvers to larger problems, where continuum solvers may not be applicable either due to noisy costs or irregular marginals with strongly fluctuating densities.
	
	Future work will comprise a more detailed study of other cost functions and more efficient implementations for hierarchical bounds needed for the construction of shielding neighbourhoods.
	
	\textbf{	Acknowledgements.} The author gratefully acknowledges support by a public grant overseen by the French National Research Agency (ANR) as part of the `Investissements d'avenir', program-reference ANR-10-LABX-0098 and the European Research Council (project SIGMA-Vision).

\appendix

\section{Additional Proofs}
\label{apx:Proofs}
\paragraph{Proof of \thref{thm:ShieldingPsiPEucl}.}
	For $y \in \hcellY$ write $y = \rep(\hcellY) + \delta$ with $|\delta|\leq \rad(\hcellY)$. $h(z) = |z|^p$ is differentiable on $\R^n$ with
	\begin{align*}
		\partial h(z) = \{\nabla h(z) \} = \{ p\,|z|^{p-1}\,n(z) \}
	\end{align*}
	where $n(z)$ denotes normalizing the vector $z$ to unit length. The ambiguity of $n(z)$ at $z=0$ is irrelevant as $|z|^{p-1}=0$ in that case. So for convenience we define $n(0)$ to be some arbitrary vector of unit length. From \eqref{eq:ShieldingConvexSubgradient} one finds:
	\begin{align*}
		\psi_{(x_A,x_s)}(y) & > p\,|x_s-\rep(\hcellY) - \delta|^{p-1} \la n(x_s-\rep(\hcellY) - \delta), x_A-x_s \ra
	\end{align*}
	We now separately bound the inner product and the absolute value term. One has:
	\begin{align*}
		\la n(x_s-\rep(\hcellY) - \delta), x_A-x_s \ra & = |x_A-x_s|\cdot \cos(\measuredangle(x_A-x_s,x_s-\rep(\hcellY) - \delta))
	\end{align*}
	We need to find an upper bound for the angle (using the triangle inequality on the 2-sphere):
	\begin{align*}
		\measuredangle(x_A-x_s,x_s-\rep(\hcellY) - \delta) & \leq \measuredangle(x_A-x_s,x_s-\rep(\hcellY)) + \measuredangle(x_s-\rep(\hcellY),x_s-\rep(\hcellY) - \delta) \\
		\measuredangle(x_s-\rep(\hcellY),x_s-\rep(\hcellY) - \delta) & \leq \begin{cases}
			\arcsin(|\delta|/|x_s-\rep(\hcellY)|) & \tn{for } |\delta| < |x_s-\rep(\hcellY)| \\
			\pi & \tn{else}
			\end{cases}
			\leq \theta \\
		\intertext{Since the maximal (unsigned) angle is $\pi$, we eventually find:}
		\measuredangle(x_A-x_s,x_s-\rep(\hcellY) - \delta) & \leq \min \{ \pi, \measuredangle(x_A-x_s,x_s-\rep(\hcellY)) + \theta \} = \varphi
	\end{align*}
	Depending on the sign of $\cos(\varphi)$ we bound $|x_s-\rep(y)-\delta|$ from above or below which yields the expression for $R$.

\paragraph{Proof of \thref{thm:ShieldingPsiSphere}.}
	Let $\kreis$ be the unit circle which we identify with the interval $(-\pi,\pi]$ `with its ends connected'.
	Denote by $F : [0,\pi] \times \kreis \rightarrow \sphere$ the map from spherical coordinates onto $\sphere$:
	\begin{align}
		\label{eq:AppendixFSpherical}
		F(\theta,\varphi) = \begin{pmatrix} \sin \theta \cdot \cos \varphi \\
			\sin \theta \cdot \sin \varphi \\
			\cos \theta
			\end{pmatrix}
	\end{align}
	For some $y \in \hcellY$ with $d(y,\rep(\hcellY)) \leq \rad(\hcellY)$ clearly have:
	\begin{align*}
		d(x_A,y) - d(x_s,y) \geq \inf \{ d(x_A,y') - d(x_s,y') : y' \in \sphere, d(y',\rep(\hcellY)) \leq \rad(\hcellY) \}
	\end{align*}
	Now find a relaxation of the feasible set on the r.h.s.\ to obtain a tractable lower bound. 
	It can be shown that the metric ball around $\rep(\hcellY)$ is contained in a sufficiently large `rectangle' in spherical coordinates. More precisely:
	\begin{align*}
		\{ y' \in \sphere, d(y',\rep(\hcellY)) \leq \rad(\hcellY) \} \subset F(\mc{D}) \\
	\end{align*}
	with
	\begin{align*}
		\mc{D} & = \mc{D}_\theta \times \mc{D}_\varphi \\
		\mc{D}_\theta & = [\max\{0,\theta_B - \rad(\hcellY)\}, \theta_B + \rad(\hcellY) ] \\
		\mc{D}_{\varphi} & = \begin{cases}
			[\varphi_B - \hat{\varphi}, \varphi_B + \hat{\varphi}] & \tn{if } \cos^2 \rad(\hcellY) > \cos^2 \theta_B \\
			 \kreis & \tn{else}
			\end{cases} \\
		\hat{\varphi} & = \arccos \sqrt{ \frac{ \cos^2 \rad(\hcellY) - \cos^ 2 \theta_B}{1-\cos^2 \theta_B}}
	\end{align*}
	where in $[\varphi_B - \hat{\varphi}, \varphi_B + \hat{\varphi}]$ we have to take into account the `wrapping around' at $-\pi$.
	The angle $\hat{\varphi}$ can be obtained by computing the distance between the point $\rep(\hcellY)$ and a `meridian' $F([0,\pi] \times \{\varphi'\})$ and demanding that this distance may not be smaller than $\rad(\hcellY)$. This yields $\hat{\varphi}$ as a minimal difference in longitude.
	
	Let $y' = F(\theta',\varphi')$. Recall that $d(x,y) = \arccos \la x,y \ra$ where $\la \cdot, \cdot \ra$ denotes the usual Euclidean inner product in $\R^3$. Then
	\begin{align*}
		d(x_A,y') - d(x_s,y') & = \theta' - \arccos(
			\sin \theta_s \cdot \sin \theta' \cdot \cos \varphi' +
			\cos \theta_s \cdot \cos \theta')\,.
	\end{align*}
	Minimize this over $F(\mc{D})$. For every $\theta' \in \mc{D}_\theta$ the minimizing $\varphi'$ is as close to $-\pi$ (in the $\kreis$-sense) as possible. For any $\varphi' \in \mc{D}_{\varphi}$ the minimizing $\theta'$ is as close to $0$ as possible. This yields $\theta_{B,\tn{min}}$ and $\varphi_{B,\tn{max}}$. Then, depending on the sign of $\Delta d_{\tn{min}}$ pick $\xi$ from the sub-differential at either the minimal or maximal end of possible distances.

\phantomsection
\addcontentsline{toc}{section}{\refname}
\bibliography{references}{}

\begin{thebibliography}{10}

\bibitem{CPLEX}
{CPLEX}.
\newblock \url{http://www.ilog.com}.

\bibitem{NetworkFlows1993}
R.~K. Ahuja, T.~L. Magnanti, and J.~B. Orlin.
\newblock {\em Network Flows: Theory, Algorithms, and Applications}.
\newblock Prentice-Hall, Inc., 1993.

\bibitem{Ambrosio2013}
L.~Ambrosio and N.~Gigli.
\newblock A user's guide to optimal transport.
\newblock In {\em Modelling and Optimisation of Flows on Networks}, volume 2062
  of {\em Lect. Not. Math.}, pages 1--155. Springer, 2013.

\bibitem{ConvexFunctionalAnalysis-11}
H.~H. Bauschke and P.~L. Combettes.
\newblock {\em Convex Analysis and Monotone Operator Theory in {H}ilbert
  Spaces}.
\newblock CMS Books in Mathematics. Springer, 1st edition, 2011.

\bibitem{BenamouBrenier2000}
J.-D. Benamou and Y.~Brenier.
\newblock A computational fluid mechanics solution to the {M}onge-{K}antorovich
  mass transfer problem.
\newblock {\em Numerische Mathematik}, 84(3):375--393, 2000.

\bibitem{BenamouIterativeBregman2015}
J.-D. Benamou, G.~Carlier, M.~Cuturi, L.~Nenna, and G.~Peyr\'e.
\newblock Iterative {B}regman projections for regularized transportation
  problems.
\newblock https://hal.archives-ouvertes.fr/hal-01096124, 2014.

\bibitem{ObermanMongeAmpere2014}
J.-D. Benamou, B.~D. Froese, and A.~M. Oberman.
\newblock Numerical solution of the optimal transportation problem using the
  {Monge–Amp\`ere} equation.
\newblock {\em Journal of Computational Physics}, 260(1):107--126, 2014.

\bibitem{OTManifoldLagrangian2007}
P.~Bernard and B.~Buffoni.
\newblock Optimal mass transportation and {Mather} theory.
\newblock {\em Journal of the European Mathematical Society}, 9(1):85--121,
  2007.

\bibitem{Bertsekas-AuctionAlgorithm1979}
D.~P. Bertsekas.
\newblock A distributed algorithm for the assignment problem.
\newblock Technical report, Lab. for Information and Decision Systems Report,
  MIT, May 1979.

\bibitem{BertsimasLP}
D.~Bertsimas and J.~N. Tsitsiklis.
\newblock {\em Introduction to Linear Optimization}.
\newblock Athena Scientific, 1997.

\bibitem{MonotoneRerrangement-91}
Y.~Brenier.
\newblock {P}olar factorization and monotone rearrangement of vector-valued
  functions.
\newblock {\em Comm.~Pure Appl.~Math.}, 44(4):375--417, 1991.

\bibitem{wassersteinRegularization2011}
M.~Burger, M.~Franek, and C.-B. Sch\"onlieb.
\newblock Regularised regression and density estimation based on optimal
  transport.
\newblock {\em Applied Mathematics Research eXpress}, 3 2012.

\bibitem{ReviewMongeMatrix-96}
R.~E. Burkhard, B.~Klinz, and R.~Rudolf.
\newblock Perspectives of {M}onge properties in optimization.
\newblock {\em Discr.~Appl.~Math.}, 70(2):95--161, 1996.

\bibitem{BrenierMap-10}
G.~Carlier, A.~Galichon, and F.~Santambrogio.
\newblock From {K}nothe's transport to {B}renier's map and a continuation
  method for optimal transport.
\newblock {\em SIAM J.~Math.~Anal.}, 41:2554--2576, 2010.

\bibitem{Cuturi2013}
M.~Cuturi.
\newblock {S}inkhorn distances: Lightspeed computation of optimal
  transportation distances.
\newblock In {\em Advances in Neural Information Processing Systems 26 ({NIPS}
  2013)}, pages 2292--2300, 2013.
\newblock http://arxiv.org/abs/1306.0895.

\bibitem{LEMON}
B.~Dezs\H{o}a, A.~J\"uttnerb, and P.~Kov\'acsa.
\newblock {LEMON} – an open source {C++} graph template library.
\newblock In {\em Proceedings of the Second Workshop on Generative Technologies
  (WGT) 2010}, volume 264 of {\em Electronic Notes in Theoretical Computer
  Science}, pages 23--45, 2011.

\bibitem{Steidl-OT-RGB-2015}
J.~H. Fitschen, F.~Laus, and G.~Steidl.
\newblock Transport between {RGB} images motivated by dynamic optimal
  transport.
\newblock http://arxiv.org/abs/1509.06142, 2015.

\bibitem{McCannGangboOTGeometry1996}
W.~Gangbo and R.~J. McCann.
\newblock The geometry of optimal transportation.
\newblock {\em Acta Math.}, 177(2):113--161, 1996.

\bibitem{GoldbergCostScaling1990}
A.~V. Goldberg and R.~E. Tarjan.
\newblock Finding minimum-cost circulations by successive approximation.
\newblock {\em Math.~Oper.~Res.}, 15(3):430--466, 1990.

\bibitem{OptimalTransportWarping}
S.~Haker, L.~Zhu, A.~Tannenbaum, and S.~Angenent.
\newblock Optimal mass transport for registration and warping.
\newblock {\em Int.~J.~Comp.~Vision}, 60:225--240, December 2004.

\bibitem{KuhnHungarianMethod}
H.~W. Kuhn.
\newblock The {H}ungarian method for the assignment problem.
\newblock {\em Naval Research Logistics}, 2:83--97, 1955.

\bibitem{TreeEMD2007}
H.~Ling and K.~Okada.
\newblock An efficient earth mover's distance algorithm for robust histogram
  comparison.
\newblock {\em IEEE Trans.~Patt.~Anal.~Mach.~Intell.}, 29(5):840--853, 2007.

\bibitem{RumpfGeneralizedOT2014}
J.~Maas, M.~Rumpf, C.~Sch\"onlieb, and S.~Simon.
\newblock A generalized model for optimal transport of images including
  dissipation and density modulation.
\newblock submitted, 2014.

\bibitem{McCannPolarManifold2001}
R.~J. McCann.
\newblock Polar factorization of maps on {R}iemannian manifolds.
\newblock {\em Geom.~Funct.~Anal.}, 11(3):589--608, 2001.

\bibitem{MultiscaleTransport2011}
Q.~M{\'e}rigot.
\newblock A multiscale approach to optimal transport.
\newblock {\em Computer Graphics Forum}, 30(5):1583--1592, 2011.

\bibitem{ObermanOptimalTransportationLP2015}
A.~M. Oberman and Y.~Ruan.
\newblock An efficient linear programming method for optimal transportation.
\newblock http://arxiv.org/abs/1509.03668.

\bibitem{Pele2009}
O.~Pele and W.~Werman.
\newblock Fast and robust {E}arth {M}over's {D}istances.
\newblock In {\em International Conference on Computer Vision ({ICCV} 2009)},
  2009.

\bibitem{GeodesicShapeMassTransport-10}
J.~Rabin, G.~Peyr\'e, and L.~D. Cohen.
\newblock Geodesic shape retrieval via optimal mass transport.
\newblock In {\em European Conference on Computer Vision ({ECCV} 2010)}, pages
  771--784, 2010.

\bibitem{RubnerEMD-IJCV2000}
Y.~Rubner, C.~Tomasi, and L.~J. Guibas.
\newblock The earth mover's distance as a metric for image retrieval.
\newblock {\em Int.~J.~Comp.~Vision}, 40(2):99--121, 2000.

\bibitem{Santambrogio-OTAM}
F.~Santambrogio.
\newblock {\em Optimal Transport for Applied Mathematicians}, volume~87 of {\em
  Progress in Nonlinear Differential Equations and Their Applications}.
\newblock Birkh\"auser Boston, 2015.

\bibitem{Schmitzer-SSVM2015}
B.~{Sch}mitzer.
\newblock A sparse algorithm for dense optimal transport.
\newblock In {\em Scale Space and Variational Methods ({SSVM} 2015)}, pages
  629--641, 2015.

\bibitem{SchmitzerSchnoerr-SSVM2013}
B.~{Sch}mitzer and C.~{Sch}n{\"o}rr.
\newblock A hierarchical approach to optimal transport.
\newblock In {\em Scale Space and Variational Methods ({SSVM} 2013)}, pages
  452--464, 2013.

\bibitem{SchmitzerSchnoerr-WassersteinModes2014}
B.~{Sch}mitzer and C.~{Sch}n{\"o}rr.
\newblock Globally optimal joint image segmentation and shape matching based on
  {Wasserstein} modes.
\newblock {\em Journal of Mathematical Imaging and Vision}, 52(3):436--458,
  2015.

\bibitem{linearApproxMassTransportCVPR2008}
S.~Shirdhonkar and D.~W. Jacobs.
\newblock Approximate earth mover’s distance in linear time.
\newblock In {\em Computer Vision and Pattern Recognition ({CVPR} 2008)}, 2008.

\bibitem{GraphMatchingGPU-EMMCVPR-2009}
C.~N. Vasconcelos and B.~Rosenhahn.
\newblock Bipartite graph matching computation on {GPU}.
\newblock In {\em Energy Minimization Methods in Computer Vision and Pattern
  Recognition ({EMMCVPR} 2009)}, pages 42--55, 2009.

\bibitem{Villani-OptimalTransport-09}
C.~Villani.
\newblock {\em Optimal Transport: Old and New}, volume 338 of {\em Grundlehren
  der mathematischen Wissenschaften}.
\newblock Springer, 2009.

\bibitem{OptimalTransportTangent2012}
W.~Wang, D.~Slep\v{c}ev, S.~Basu, J.~A. Ozolek, and G.~K. Rohde.
\newblock A linear optimal transportation framework for quantifying and
  visualizing variations in sets of images.
\newblock {\em Int.~J.~Comp.~Vision}, 101:254--269, 2012.

\end{thebibliography}
\bibliographystyle{plain}

\end{document}